\documentclass[11pt]{article}

\usepackage[
	a4paper,
	margin=1in
]{geometry}
\usepackage{setspace}

\usepackage{sectsty} 
\usepackage{titlesec}
\allsectionsfont{\boldmath} 

\usepackage{appendix}

\usepackage{amsmath, amssymb, amsthm}
\usepackage{mathtools}

\numberwithin{equation}{section}

\usepackage[
	setpagesize=false,
	colorlinks=true,
	linkcolor=BrickRed,
        citecolor=OliveGreen,
        urlcolor=black,
	pdfencoding=auto,
	psdextra,
]{hyperref}
\usepackage{cleveref}

\usepackage{etoolbox} 
\usepackage[shortlabels]{enumitem}
\usepackage{xparse} 
\usepackage{blkarray,multirow}

\usepackage{graphicx}
\usepackage[dvipsnames]{xcolor}
\usepackage{tikz}
\usetikzlibrary{patterns}
\usetikzlibrary{decorations.markings, arrows.meta}

\usepackage{todonotes}
\usepackage[
    deletedmarkup=sout,
    commentmarkup=uwave,
    authormarkuptext=name
]{changes}

\usepackage{comment}

\usepackage{tocloft}

\setlist[enumerate,1]{label=(\arabic*), ref=(\arabic*)}
\setlist[enumerate,3]{label=(\roman*), ref=(\roman*)}

\theoremstyle{plain}
\newtheorem{theorem}{Theorem}[section]
\newtheorem{lemma}[theorem]{Lemma}
\newtheorem{corollary}[theorem]{Corollary}
\newtheorem{proposition}[theorem]{Proposition}

\newtheorem{conjecture}[theorem]{Conjecture}
\newtheorem{problem}[theorem]{Problem}

\newtheorem{claim}[theorem]{Claim}
\newtheorem*{claim*}{Claim}

\makeatletter
\newenvironment{claimproof}[1][Proof]{\par
	\pushQED{\qed}%
	
	\normalfont \topsep6\p@\@plus6\p@\relax
	\trivlist
	\item[\hskip\labelsep
	\textit{#1}\@addpunct{.}~]\ignorespaces
}{%
	\popQED\endtrivlist\@endpefalse
}
\makeatother

\theoremstyle{definition}
\newtheorem{definition}[theorem]{Definition}
\newtheorem{remark}[theorem]{Remark}

\newcommand{\calC}{\mathcal{C}}
\newcommand{\calD}{\mathcal{D}}
\newcommand{\calF}{\mathcal{F}}

\newcommand{\calH}{\mathcal{H}}

\newcommand{\calP}{\mathcal{P}}
\newcommand{\calQ}{\mathcal{Q}}

\newcommand{\calW}{\mathcal{W}}
\newcommand{\calX}{\mathcal{X}}
\newcommand{\calY}{\mathcal{Y}}
\newcommand{\calZ}{\mathcal{Z}}

\newcommand{\ve}{\varepsilon}

\newcommand{\ori}[1]{\overrightarrow{#1}}
\newcommand{\dir}[1]{\overleftrightarrow{#1}}

\newcommand{\Aut}{\mathrm{Aut}}

\newcommand{\defeq}{\coloneqq}

\title{Kahn--Lov\'{a}sz-type inequalities for graph factors}

\author{
Hyunwoo Lee%
        \thanks{Department of Mathematical Sciences, KAIST, South Korea and Extremal Combinatorics and Probability Group (ECOPRO), Institute for Basic Science (IBS).
        E-mail: {\ttfamily hyunwoo.lee@kaist.ac.kr.} Supported by the National Research Foundation of Korea (NRF) grant funded by the Korea government(MSIT) No. RS-2023-00210430, and the Institute for Basic Science (IBS-R029-C4).}
}

\usepackage[square,sort,comma,numbers]{natbib}
\begin{document}
\maketitle

\begin{abstract}
    The Kahn--Lov\'{a}sz theorem gives a sharp upper bound on the number of perfect matchings in a graph in terms of its degree sequence, extending the classical Br\'{e}gman--Minc inequality for bipartite graphs. In this paper, we establish an asymptotically sharp extension of the Kahn--Lov\'{a}sz theorem to $F$-factors for every Hamiltonian graph $F$. As a consequence, we asymptotically determine the maximum number of $F$-factors in an $n$-vertex $m$-edge graph, yielding an $F$-factor analogue of Kruskal--Katona-type theorems.

    We also prove a multigraph analogue of the Kahn--Lov\'{a}sz theorem. Combining this with our results for Hamiltonian graphs, we obtain an asymptotically sharp Kruskal--Katona-type bound for a further class of connected graphs $F$, including those containing two vertex-disjoint cycles of equal length whose union spans $V(F)$.
\end{abstract}


\setlength{\cftbeforesecskip}{3pt}
\setlength{\cftbeforesubsecskip}{0pt}
\tableofcontents


\section{Introduction}\label{sec:intro}

The \emph{permanent} of an $n \times n$ matrix $A$, denoted by $\mathrm{per}(A)$, is a fundamental quantity in combinatorial matrix theory and in extremal and enumerative combinatorics. It is defined by
\begin{equation}\label{eq:perm-def}
    \mathrm{per}(A) \defeq \sum_{\sigma \in S_n} \prod_{i \in [n]} A_{i,\sigma(i)},
\end{equation} 
where $S_n$ is the symmetric group on $n$ elements.
The permanent of a $0$--$1$ matrix is of particular interest since it coincides with the number of perfect matchings in an associated bipartite graph. Let $G$ be a bipartite graph with bipartition $L \cup R$, where $L = \{\ell_1, \dots, \ell_n\}$ and $R = \{r_1, \dots, r_n\}$. Define the $0$--$1$ matrix $B(G)$ by setting $B(G)_{i, j} = 1$ if and only if $\ell_i r_j \in E(G)$ for all $i, j \in [n]$. 
Then by \eqref{eq:perm-def}, the permanent of $B(G)$ is precisely the number of perfect matchings in $G$. We call $B(G)$ a \emph{bipartite adjacency matrix} of $G$. 

This connection implies that estimating the number of perfect matchings in a bipartite graph is equivalent to estimating the permanent of its bipartite adjacency matrix. In general, designing an efficient algorithm for computing the exact value of the permanent is one of the central challenges in theoretical computer science. Indeed, even when restricted to $0$--$1$ matrices, the problem remains computationally intractable since Valiant~\cite{Valiant} proved that computing the permanent of a $0$--$1$ matrix is $\#P$-complete. Because computing the permanent exactly is difficult in general, considerable effort has been devoted to establishing general upper and lower bounds on the permanent.

In 1963, Minc~\cite{Minc} conjectured an upper bound on the permanent of a $0$--$1$ matrix with prescribed row sums. This conjecture was famously resolved by Br\'{e}gman~\cite{Bregman} in 1973. Since the permanent of a bipartite adjacency matrix counts the number of perfect matchings in the bipartite graph, the Br\'{e}gman--Minc inequality can be stated in the following graph-theoretic form. For graphs $F$ and $G$, we denote by $N_{\mathrm{factor}}(F; G)$ the number of distinct $F$-factors in $G$. Note that a $K_2$-factor is equivalent to a perfect matching.

\begin{theorem}[Br\'{e}gman--Minc inequality]\label{thm:bregman-minc}
    Let $G$ be a bipartite graph on bipartition $V(G) = L \cup R$ with $|L| = |R|$. Then we have
    $$
        N_{\mathrm{factor}}(K_2; G) \leq \prod_{v\in L} (d_G(v)!)^{\frac{1}{d_G(v)}},
    $$ where $d_G(v)$ is the degree of $v$ in $G$.
\end{theorem}

As bipartite graphs and perfect matchings are fundamental objects in graph theory, Theorem~\ref{thm:bregman-minc} has numerous applications to bounding the number of large combinatorial structures, including balanced orientations of a graph~\cite{Schrijver-orientation}, directed Hamilton paths in a tournament~\cite{Alon-Ham-path}, and Hamilton decompositions of a graph~\cite{Count-Ham-decomp,Count-di-Ham-decomp}.

A natural extension of \Cref{thm:bregman-minc} is to allow the host graph to be non-bipartite. Kahn and Lov\'{a}sz proved the following extension, although their proof was not published; see~\cite{Kahn-Lovasz-ref}. Since then, the Kahn--Lov\'{a}sz theorem has been rediscovered and reproved several times~\cite{KL-1,KL-2,KL-3}.

\begin{theorem}[Kahn--Lov\'{a}sz]\label{thm:kahn-lovasz}
    Let $G$ be a graph. Then we have 
    $$
        N_{\mathrm{factor}}(K_2; G) \leq  \prod_{v\in V(G)} (d_G(v)!)^{\frac{1}{2 d_G(v)}}.
    $$
\end{theorem}

We remark that the inequalities in Theorems~\ref{thm:bregman-minc} and~\ref{thm:kahn-lovasz} are sharp, with equality attained when $G$ is a disjoint union of complete balanced bipartite graphs.

As a simple corollary, the celebrated Kruskal--Katona theorem~\cite{Kruskal,Katona} gives a sharp upper bound on the number of copies of $K_r$ in an $n$-vertex $m$-edge graph. For an integer $m$, let $u$ be the nonnegative real number satisfying $\binom{u}{2}=m$. Then every $n$-vertex $m$-edge graph contains at most $\binom{u}{r}$ copies of $K_r$, and this bound is sharp~\cite{Lovasz-KK}. Motivated by this, Alon~\cite{Alon-KK} determined, up to a multiplicative constant, the maximum number of copies of $F$ in an $m$-edge graph for every fixed graph $F$. Since then, various extensions of this problem under additional restrictions on the host graph have been studied; see, for instance,~\cite{Gerbner-Nagy-Patkos-Vizer,Kirsch-Radcliffe,Chakraborti-Chen,Chao-Yu}. We call a problem or theorem \emph{Kruskal--Katona-type} if it concerns the maximum number of copies of a specified subgraph, which need not be fixed or small, in a graph with a prescribed number of edges.

Combining \Cref{thm:kahn-lovasz} with Stirling's approximation and the AM-GM inequality yields the following Kruskal--Katona-type result for perfect matchings.

\begin{corollary}\label{cor:KK-kahn-lovasz}
    For all $\ve > 0$, there exists $C > 0$ such that for all even $n > 0$ and $m$ with $m \geq Cn$, we have
    $$
        \max_{\substack{G: |V(G)| = n, \\|E(G)| = m}} N_{\mathrm{factor}}(K_2; G) = \left[(1 \pm \ve)\left(\frac{2m}{e n} \right)\right]^{\frac{n}{2}},
    $$ where $e$ is Euler's number.
\end{corollary}

In this paper, we extend the Kahn--Lov\'asz theorem to $F$-factors and establish Kruskal--Katona-type results for $F$-factors. As an application, we obtain nontrivial upper bounds on the number of distinct cycle decompositions in terms of the degree sequence of the host graph. These bounds are asymptotically sharp for many graphs.


\subsection{Main results}\label{subsec:Results}

Our first main result extends \Cref{thm:kahn-lovasz} to $F$-factors for several classes of graphs $F$. In particular, we obtain an asymptotically sharp result for every Hamiltonian graph $F$.

\begin{theorem}\label{thm:main-factors}
    For every $\ve > 0$ and a Hamiltonian graph $F$, there exists a constant $C = C(F, \ve) > 0$ such that the following holds. For all $n$-vertex graph $G$, we have
    \begin{equation}\label{eq:main-equation}
        N_{\mathrm{factor}}(F; G) \leq \left( (1  + \ve) \frac{|V(F)|}{|\mathrm{Aut}(F)|} \right)^{\frac{n}{|V(F)|}} \cdot \left( \prod_{v\in V(G)} \frac{d_G(v) + C}{e} \right)^{1 - \frac{1}{|V(F)|}},
    \end{equation} where $\mathrm{Aut}(F)$ denotes the automorphism group of $F$.
\end{theorem}

The bound in \Cref{thm:main-factors} is asymptotically sharp. The extremal construction is given by disjoint unions of cliques satisfying suitable divisibility conditions; see \Cref{sec:lowerbounds}. 

\begin{remark}
    Unfortunately, the inequality \eqref{eq:main-equation} fails for some non-Hamiltonian graphs; highly unbalanced bipartite graphs provide examples. Nevertheless, in Section~\ref{sec:F-factors}, we establish a nontrivial upper bound on $N_{\mathrm{factor}}(F; G)$ in terms of the degree sequence of the host graph for every connected graph $F$. We refer the reader to Theorem~\ref{thm:general-factors}.
\end{remark} 

Applying the AM-GM inequality to \Cref{thm:main-factors}, we also obtain an asymptotically sharp Kruskal--Katona-type result for $F$-factors whenever $F$ is Hamiltonian. To state this result, let $N_F(n,m)$ denote the maximum number of $F$-factors in an $n$-vertex $m$-edge graph.

\begin{corollary}\label{cor:factor-kruskal-katona}
    For every $\varepsilon > 0$ and a Hamiltonian graph $F$, there exists a constant $C = C(F, \ve) > 0$ such that the following holds. 
    For all positive integers $n$ and $m$ with $n$ divisible by $|V(F)|$ and $m \geq Cn$, we have
    $$
        N_F(n, m) \leq \left[(1 + \ve) \left( \frac{|V(F)|}{|\mathrm{Aut}(F)|} \right)^{\frac{1}{|V(F)| - 1}} \frac{2m}{e n} \right]^{\left(1 - \frac{1}{|V(F)|}\right) n}.
    $$
    Moreover, this bound is asymptotically sharp for infinitely many pairs $(n,m)$.
\end{corollary}

The lower bound of Corollary~\ref{cor:factor-kruskal-katona} is attained by disjoint unions of cliques, as described in \Cref{sec:lowerbounds}.

We also prove an analogue of Corollary~\ref{cor:factor-kruskal-katona} for graphs $F$ that contain two vertex-disjoint cycles of equal length whose union spans $V(F)$. This follows from Theorem~\ref{thm:main-factors} and a multigraph analogue of the Kahn--Lov\'{a}sz theorem, established in Section~\ref{sec:multigraph}.

\begin{theorem}\label{thm:two-cycles}
    Let $F$ be a connected graph that contains two vertex-disjoint cycles of equal length whose union spans $V(F)$. Then for every $\ve > 0$, there exists a constant $C = C(F, \ve) > 0$ such that the following holds. 
    For all positive integers $n$ and $m$ with $n$ divisible by $|V(F)|$ and $m \geq Cn$, we have
    $$
        N_F(n, m) \leq \left[(1 + \ve) \left( \frac{|V(F)|}{|\mathrm{Aut}(F)|} \right)^{\frac{1}{|V(F)| - 1}} \frac{2m}{e n} \right]^{\left(1 - \frac{1}{|V(F)|}\right) n}.
    $$
    Moreover, this bound is asymptotically sharp for infinitely many pairs $(n,m)$.
\end{theorem}

We note that there exist many interesting non-Hamiltonian graph but contain a spanning disjoint union of two cycles of equal length, for instance, the Petersen graph. In other words, Theorem~\ref{thm:main-factors} and Corollary~\ref{cor:factor-kruskal-katona} do not apply to the Petersen graph, but Theorem~\ref{thm:two-cycles} works for the Petersen graph.


\subsection{High-level overview}\label{subsec:overview}
We now give a high-level overview of the main ideas behind our proofs.

We first discuss Theorem~\ref{thm:main-factors}. Let $F$ be a Hamiltonian graph on $\ell$ vertices. Since $F$ contains a spanning cycle $C_\ell$, a double-counting argument reduces the problem to bounding the number of embeddings of $C_\ell$-factors into a given host graph. To prove it, we partition $V(G)$ into $\ell$ equally sized classes $V_0,\dots,V_{\ell-1}$ and count only those cycle factors whose vertices appear successively in these classes. If we ignore the cycle edges between one fixed pair of consecutive classes, the remaining edges form perfect matchings between the other consecutive pairs. Thus, the Br\'{e}gman--Minc inequality (Theorem~\ref{thm:bregman-minc}), gives an upper bound in terms of the numbers of neighbors that each vertex has in the next class.

We apply this estimate for each of the $\ell$ possible pairs of consecutive classes and take the geometric mean. Since every vertex contributes to exactly $\ell-1$ of the resulting estimates, this produces the exponent $1-1/\ell$ in Theorem~\ref{thm:cycle-factor}, which corresponds to the exponent $1 - 1/|V(F)|$ in Theorem~\ref{thm:main-factors}. We then sum over all equal partitions and apply H\"{o}lder's inequality. The remaining task is to control a weighted sum over partitions, which is precisely the content of Lemma~\ref{lem:partition-cycle-sum}.

The main idea behind Lemma~\ref{lem:partition-cycle-sum} is to reverse the order of counting. Recall that the Br\'egman--Minc inequality naturally gives a bound involving factorials of the relevant degrees. 
Using Stirling's approximation, in particular Proposition~\ref{prop:factorial-linear}, we replace these factorial terms, up to an arbitrarily small multiplicative error, by linear functions of the degrees. 
This linearization is crucial, as it allows us to interpret the resulting product in purely graph-theoretic terms. More precisely, replace every edge of $G$ by the two possible orientations and add a bounded number of directed loops at each vertex. 
For a fixed partition, the resulting product counts the number of ways to choose exactly one outgoing edge from every vertex, subject to the condition that each chosen non-loop edge goes from one class to the next. We instead first fix the resulting directed graph and count the number of partitions compatible with it.

Observe that each connected component of such a directed graph contains a unique directed cycle, and once the class of one vertex in a component is fixed, the classes of all other vertices are determined. 
Hence, if the directed graph has $k$ connected components, it is compatible with at most $\ell^k$ partitions. 
Lemma~\ref{lem:exponential-count} shows that this additional weight can be absorbed into an arbitrarily small exponential loss: graphs with few components have small weight, while graphs with many components necessarily contain many small components and are correspondingly sparse. 
This proves Lemma~\ref{lem:partition-cycle-sum}, and hence Theorem~\ref{thm:cycle-factor} and Theorem~\ref{thm:main-factors}.

The proof of Theorem~\ref{thm:general-factors}, the general connected graph case, follows the same strategy. We choose a spanning tree $T$ of $F$ and reduce the problem to Theorem~\ref{thm:tree-factor}. After fixing an equal partition of $V(G)$, we root $T$ at one of its vertices and orient every edge towards the root. The resulting structure can again be counted using perfect matchings between appropriate pairs of classes. The Br\'egman--Minc inequality, geometric averaging over the choices of the root, and H\"older's inequality then give the desired bound. The additional factor involving $a^a b^b$ comes from the two color classes of $T$, of sizes $a$ and $b$.

Finally, to prove Theorem~\ref{thm:two-cycles}, we first count factors consisting of cycles using Corollary~\ref{cor:factor-kruskal-katona}. For a fixed cycle factor $H$, we construct an auxiliary loopless multigraph whose vertices correspond to the cycles of $H$, with multiplicities given by the numbers of edges of $G$ joining pairs of cycles. Perfect matchings in this multigraph correspond to ways to pair the cycles and add one joining edge to each pair. We therefore prove the multigraph Kahn--Lov\'asz theorem, Theorem~\ref{thm:multigraph-KL}, using the entropy method. When the multiplicities are bounded, and the average degree is sufficiently large, the resulting error term is negligible, giving Corollary~\ref{cor:multigraph-KK}. Applying this to the auxiliary multigraph and then using Lemma~\ref{lem:aut-double-counting} yields Theorem~\ref{thm:two-cycles}.


\paragraph{Organization.}
The rest of the paper is organized as follows. In Section~\ref{sec:prelim}, we introduce the notation used throughout the paper and briefly review some basic properties of information entropy. In Section~\ref{sec:lowerbounds}, we establish the lower bounds for our main results. Section~\ref{sec:F-factors} contains the proof of Theorem~\ref{thm:main-factors}, together with its extension to more general cases. As mentioned above, Theorem~\ref{thm:two-cycles} follows from Theorem~\ref{thm:main-factors} and a multigraph analogue of the Kahn--Lov\'{a}sz theorem. In Section~\ref{sec:multigraph}, we establish this multigraph extension and then prove Theorem~\ref{thm:two-cycles}. We conclude the paper with some remarks and open problems.


\phantomsection
\addcontentsline{toc}{subsection}{Statement of AI use}
\paragraph{Statement of AI use.}
We used ChatGPT 5.6 Pro/Sol to improve the exposition and proofread the manuscript. Apart from this, no AI tools were used. All mathematical ideas were developed independently by the author, who takes full responsibility for this article.


\section{Preliminaries}\label{sec:prelim}

\subsection{Notation}\label{subsec:notation}
We write $[n]=\{1,\dots,n\}$ for each positive integer $n$. We use the notation
$\beta \ll \alpha_1,\dots,\alpha_t$ to mean that there exists a function $f$ such that $\beta \leq f(\alpha_1,\dots,\alpha_t)$. We do not attempt to specify the function $f$ explicitly.

Throughout the paper, we use standard notation from graph theory.  All graphs and digraphs are considered simple unless otherwise stated. 
For a graph $G$, we denote its vertex set and edge set by $V(G)$ and $E(G)$, respectively, and write $v(G) \defeq |V(G)|$ and $e(G) \defeq |E(G)|$. For a vertex $v \in V(G)$, we denote by $N_G(v)$ and $d_G(v)$ the set of neighbors and degree of $v$ in $G$, respectively. We often omit the subscript when the underlying graph is clear from the context.
For a digraph $D$ and a vertex $v \in V(D)$, we use analogous notation. In particular, we denote by $N^+_D(v)$ and $N^-_D(v)$ be out- and in-neighbors of $v$ in $D$, and refer their sizes as out- and in-degree of $v$, namely $d^+_D(v)$ and $d^-_D(v)$, respectively. Again, we omit the subscript when the underlying digraph is clear from the context. For a graph $G$ and two disjoint vertex subsets $X, Y\subseteq V(G)$, we write $G[X]$ the induced subgraph of $G$ induced by $X$, and $G[X, Y]$ the induced bipartite subgraph of $G$, where the edges of $G[X, Y]$ the edges whose endpoints are lies in both $X$ and $Y$.

In this paper, we also consider a loopless multigraph and a digraph with multi-loops in the proof of our theorems. We also consider loopless multigraphs and digraphs with multiple loops. For a loopless multigraph $M$ and a pair $e\in\binom{V(M)}{2}$, let $\mu(e)$ denote the multiplicity of $e$, and let $\mu(M)$ denote the maximum edge multiplicity of $M$. For $v\in V(M)$, we define $d_M(v)\defeq\sum_{u\in V(M)}\mu(uv)$, that is, the number of edges incident with $v$, counted with multiplicity. 
We say a digraph $D$ has multi-loops if some vertices of $D$ have multiple loops directed out and also into themselves. 
We write $\mathrm{Comp}(G)$ and $\mathrm{Comp}(D)$ for the set of connected components of $G$ and the set of connected components of the underlying graph of $D$, respectively.
Note that all simple graphs and digraphs are also loopless multigraphs and digraphs with multi-loops.

For graphs $F$ and $G$, we denote by $N(F; G)$ the number of $F$ copies in $G$. We write $m \cdot F$ for the disjoint union of $m$ copies of $F$ for a positive integer $m$. With these notation, $N_{\mathrm{factor}}(F; G)$ is the same with $N\left(\frac{|V(G)|}{|V(F)|} \cdot F; G\right)$ whenever $|V(F)|$ divides $|V(G)|$. We say a function $\phi: V(F) \to V(G)$ is an \emph{embedding} if $\phi$ is an injective homomorphism from $F$ to $G$. We write $\mathrm{Emb}(F; G)$ for the set of embeddings from $F$ to $G$. We note that 
\begin{equation}\label{eq:embedding-automorphism}
    N(F; G) = \frac{|\mathrm{Emb}(F; G)|}{|\Aut(F)|}
\end{equation}
holds, where $\Aut(F)$ is the automorphism group of $F$.

We now introduce a new notation regarding the degree of a vertex in a given graph with respect to a certain vertex partition. Let $G$ be a graph and $v$ be a vertex of $G$. 
We define an \emph{$r$-partition} $\calP$ of $V(G)$ as an $r$-tuple $(V_0, V_1, \dots, V_{r-1})$, where the indices are members of the cyclic group $\mathbb{Z}_{r}$ such that $V_i \cap V_j = \emptyset$ whenever $i\neq j$ and $V(G) = \cup_{i\in \mathbb{Z}_r} V_i$. We say this partition is \emph{$r$-equipartition} if $|V_i| = |V_j|$ for all $i, j \in \mathbb{Z}_r$. Consider a labelled digraph $D$ on the vertex set $\mathbb{Z}_r$. Without loss of generality, assume $v\in V_0$. Then we define the out-degree of $v$ with respect to the pair $(\calP, D)$, denoted by $d^+_G(v; (\calP, D))$, as
$$
    d^+_G(v; (\calP, D)) \defeq \sum_{i \in N^+_D(0)} |N_G(v) \cap V_i|.
$$
The notion of $d^+_G(v; (\calP, D))$ plays a crucial role in the proof of Theorem~\ref{thm:main-factors}.

Lastly, throughout the paper, all the logarithms are taken base $e$ unless we indicate the base, and we consider $\log 0 = 0$.


\subsection{Entropy basics}\label{subsec:entropy-basic}

We use an entropy method inspired by Radhakrishnan's elegant proof~\cite{Radhakrishnan} of the Br\'{e}gman--Minc inequality to prove Theorem~\ref{thm:multigraph-KL} in Section~\ref{sec:multigraph}.
Information entropy (Shannon's entropy) is a very useful quantity to capture the randomness of a given discrete random variable that was introduced by Shannon in 1948. 
For a random variable $Z$ on a probability space $\calZ$ and $z \in \calZ$, we denote by $P_Z(z)$ the value $\Pr[Z = z]$.
The following is the definition of information entropy. 

\begin{definition}\label{def:entropy}
    Let $X$ be a discrete random variable on a finite set $\calX$. Then the entropy of $X$, denoted as $H(X)$ is defined by
    $$
        H(X) \defeq - \sum_{x\in \calX} P_X(x) \log P_X(x).
    $$
\end{definition}

We also need a conditional entropy defined as follows.

\begin{definition}\label{def:conditional-entropy}
    Let $X$ and $Y$ be discrete random variables on a finite set $\calX$ and $\calY$, respectively. Then the conditional entropy of $X$ and $Y$, denoted as $H(X | Y)$ is defined by
    $$
         H(X | Y) \defeq - \sum_{(x, y)\in \calX \times \calY} P_{(X, Y)}(x, y) \log \frac{P_{(X, Y)}(x, y)}{P_Y(y)}.
    $$
\end{definition}
Equivalently, $H(X | Y) = \mathbb{E}_Y[H(X|Y = y)]$.

Rather than its definition itself, entropy has been used in many combinatorial problems because of its highly useful properties. We summarize basic properties of entropy that we need below.

\begin{proposition}\label{prop:entropy}
    Let $X, Y, X_1, \dots, X_n$ be discrete random variables. Then the following properties hold.
    \begin{enumerate}
        \item[$\bullet$] \textbf{[Maximality of the uniform]} We have $0 \leq H(X) \leq \log |\calX|$, and the equality holds if and only if $X$ is a uniform random variable on $\calX$.
        \item[$\bullet$] \textbf{[Monotonicity]}  If $Y = f(X)$ for some function $f$, then $H(Y) \leq H(X)$.
        \item[$\bullet$] \textbf{[Chain rule]} $H[(X_1, \dots, X_n)] = \sum_{i\in [n]} H(X_i | X_1, \dots, X_{i-1})$. 
    \end{enumerate}
\end{proposition}

For more properties on information entropy and its applications in combinatorics, we refer the readers to a survey of Galvin~\cite{Galvin}.


\subsection{Factorial estimates}\label{subsec:numeric}

As we discussed in Section~\ref{subsec:overview}, our proof uses the Br\'{e}gman--Minc inequality and the Kahn--Lov\'{a}sz theorem as black-box theorems. Both of them contain many factorials in the statements, and we approximate them with a polynomial to connect those quantities to the number of certain digraphs. To achieve this, we collect upper and lower bounds of factorials via Stirling's approximation. Below is Stirling's approximation.

\begin{proposition}[Stirling's approximation]
    For a positive integer $n$, we have
    \begin{equation}\label{eq:stirling}
        \sqrt{2\pi n} \left( \frac{n}{e} \right)^n \leq n! \leq  \sqrt{2\pi n} \left( \frac{n}{e} \right)^n e^{\frac{1}{12 n}}.
    \end{equation}
\end{proposition}

As a consequence of the above inequality, we obtain the following useful inequality.

\begin{proposition}\label{prop:factorial-linear}
    For every real number $\ve > 0$, there exists $C = C(\ve) > 0$ such that the following holds.
    For every positive integer $n$, we have
    $$
        \frac{n}{e} \leq (n!)^{\frac{1}{n}} \leq (1 + \ve)\frac{n + C}{e}.
    $$
\end{proposition}

\begin{proof}[Proof of Proposition~\ref{prop:factorial-linear}]
    The lower bound directly comes from \eqref{eq:stirling}. Since the function $f(n) = \left(\sqrt{2\pi n} \cdot e^{\frac{1}{12 n}}\right)^{1/n}$ is decreasing and tends to $1$ as $n$ goes to infinity, there exists a natural number $k_{\ve} > 0$ such that for all integers $n > k_{\ve}$ satisfies $f(n) \leq 1 + \ve$. Hence, for this regime, we have the desired inequality by \eqref{eq:stirling}. We now assume $n \leq k_{\ve}$. Then we have $(n!)^{1/n} \leq (n^n)^{1/n} = n \leq k_{\ve}$. Thus, by choosing the constant $C$ as $e \cdot k_{\ve}$, we obtain the desired inequality. This completes the proof.
\end{proof}


\section{Sharpness constructions}\label{sec:lowerbounds}

In this section, we prove the lower bounds of Corollary~\ref{cor:factor-kruskal-katona} and Theorem~\ref{thm:two-cycles}. We also discuss lower bounds for general $F$-factor cases for connected graphs $F$.

We note that for every graph $F$, the number of $F$ copies in $K_{|V(F)|}$ is $\frac{|V(F)|!}{|\Aut(F)|}$. Therefore, for an integer $c \geq 1$, we have
\begin{equation}\label{eq:factors-in-clique}
    N_{\mathrm{factor}}(F; K_{c |V(F)|}) = \frac{(c|V(F)|)!}{c! \cdot (|V(F)|!)^{c}} \cdot \left( \frac{|V(F)|!}{|\Aut(F)|} \right)^c = \frac{(c|V(F)|)!}{c! \cdot |\Aut(F)|^c}
\end{equation}
Let $n$ be a positive integer divisible by $c|V(F)|$ and let $G$ be an $n$-vertex graph which is $\frac{n}{c|V(F)|} \cdot K_{c|V(F)|}$. Then $|E(G)| = (c|V(F)|-1)n/2 =: m$. By \eqref{eq:factors-in-clique}, we have
\begin{equation*}
    N_{\mathrm{factor}}(F; G) = \left( \frac{(c|V(F)|)!}{c! \cdot |\Aut(F)|^c}\right)^{\frac{n}{c|V(F)|}} \geq \left(\frac{1}{|\Aut(F)|} \right)^{\frac{n}{|V(F)|}} \cdot \left( \frac{(c|V(F)|)!}{c!}\right)^{\frac{n}{c|V(F)|}}.
\end{equation*}
Thus, for large enough $c > 0$, Proposition~\ref{prop:factorial-linear} implies that
\begin{equation*}
    N_{\mathrm{factor}}(F; G) \geq \left((1 - \ve)\left(\frac{|V(F)|}{|\Aut(F)|}\right) \right)^{\frac{n}{|V(F)|}} \cdot \left( \frac{c|V(F)|}{e} \right)^{\left(1 - \frac{1}{|V(F)|} \right)n}.
\end{equation*}
Since $2m/n = c|V(F)|-1 \leq c|V(F)|$, we summarize as follows.

\begin{proposition}\label{prop:lower-disj-cliques}
    For a given real number $\ve > 0$, there are infinitely many pairs of $(n, m)$ such that
    \begin{equation}\label{eq:lower-factor}
        N_F(n, m) \geq \left[(1 - \ve) \left( \frac{|V(F)|}{|\mathrm{Aut}(F)|} \right)^{\frac{1}{|V(F)| - 1}} \frac{2m}{e n} \right]^{\left(1 - \frac{1}{|V(F)|}\right) n}.
    \end{equation}
\end{proposition}

\begin{remark}
    By the monotonicity of $N_F(n, m)$ with respect to $m$ and a slight modification of the aforementioned construction, \eqref{eq:lower-factor} holds for $\Omega(n) \leq m \leq o(n^2)$. We omit the proof. 
\end{remark}

Note that, since Corollary~\ref{cor:factor-kruskal-katona} is a direct consequence of Theorem~\ref{thm:main-factors} and Proposition~\ref{prop:lower-disj-cliques} shows that Corollary~\ref{cor:factor-kruskal-katona} is asymptotically sharp, Theorem~\ref{thm:main-factors} is asymptotically sharp as well.

We now provide another construction for the number of $F$-factors, which is better than \eqref{eq:lower-factor} for some non-Hamiltonian graphs $F$. The basic idea is to consider an optimal proper coloring of $F$. Let $\chi(F) = r$ and $a_1, \dots, a_r$ be the size of each color class of $F$, respectively. For an integer $c > 0$, let $H \defeq K_{ca_1, \dots, ca_r}$ be the complete $r$-partite. Then we take $\frac{n}{c|V(F)|}\cdot H$ as a host graph. Then this construction gives a better bound than \eqref{eq:lower-factor} for some cases, including highly unbalanced bipartite graphs. For simplicity, we only demonstrate the case of unbalanced bipartite graphs.

Let $F$ be a connected bipartite graph on the bipartition sizes $a$ and $b$ where $a \neq b$. Observe that the number of $F$ copies in $K_{a, b}$ is $\frac{a!b!}{|\Aut(F)|}$ since $a\neq b$. Thus, for an integer $c \geq 1$, we have
\begin{equation}\label{eq:factors-in-biclique}
    N_{\mathrm{factor}}(F; K_{ca, cb}) = \frac{(ca)!(cb)!}{c! \cdot (a! b!)^{c}} \cdot \left( \frac{a!b!}{|\Aut(F)|} \right)^c = \frac{(ca)!(cb)!}{c! \cdot |\Aut(F)|^c}.
\end{equation}
Let $n$ be a positive integer divisible by $c(a + b)$ and let $G$ be an $n$-vertex graph which is $\frac{n}{c(a + b)} \cdot K_{ca, cb}$. Then $|E(G)| = \frac{cab}{a+b}n =: m$. Without loss of generality, assume $a< b$ and $r\defeq b/a - 1$. Then by \eqref{eq:factors-in-biclique}, it holds that
\begin{equation*}
    N_{\mathrm{factor}}(F; G) = \left( \frac{(ca)!(cb)!}{c! \cdot |\Aut(F)|^c} \right)^{\frac{n}{c(a + b)}} \geq \left(\frac{1}{|\Aut(F)|} \right)^{\frac{n}{|V(F)|}} \cdot \left( \frac{(ca)!(cb)!}{c!}\right)^{\frac{n}{c|V(F)|}}.
\end{equation*}
For large enough $c$, by Proposition~\ref{prop:factorial-linear}, we have
\allowdisplaybreaks
\begin{align*}
    N_{\mathrm{factor}}(F; G) &\geq \left(\frac{1}{|\Aut(F)|} \right)^{\frac{n}{|V(F)|}} \cdot \left( \frac{(ca)!(cb)!}{c!}\right)^{\frac{n}{c|V(F)|}}\\
    &\geq \left( \frac{(1-\ve)}{|\Aut(F)|} \right)^{\frac{n}{|V(F)|}} \cdot \left( a^a b^b \right)^{\frac{n}{|V(F)|}} \left( \frac{c}{e} \right)^{\left(1 - \frac{1}{|V(F)|} \right)n}\\
    &= \left((1 - \ve)\left( \frac{|V(F)|}{|\Aut(F)|}\right) \right)^{\frac{n}{|V(F)|}} \cdot \left( \frac{2cab}{e(a+b)} \right)^{\left(1 - \frac{1}{|V(F)|} \right)n}\\
    &\qquad \qquad \cdot \left(2^{-|V(F)| + 1}\cdot (a+b)^{|V(F)|-2} \cdot a^{a - |V(F)| + 1} \cdot b^{b - |V(F)| + 1} \right)^{\frac{n}{|V(F)|}}\\
    &= \left((1 - \ve)\left( \frac{|V(F)|}{|\Aut(F)|}\right) \right)^{\frac{n}{|V(F)|}} \cdot \left( \frac{2cab}{e(a+b)} \right)^{\left(1 - \frac{1}{|V(F)|} \right)n}\\
    &\qquad \qquad \cdot \left(2^{-(2 + r)a + 1} \cdot (2+r)^{(2+r)a - 2} \cdot (1 + r)^{-a + 1} \right)^{\frac{n}{|V(F)|}}
\end{align*}
Observe that for all large enough $a$ and sufficiently large $r$ compared with $a$, we have
$$
    2^{-(2 + r)a + 1} \cdot (2+r)^{(2+r)a - 2} \cdot (1 + r)^{-a + 1} > 1.
$$
Since $2m/n = \frac{2cab}{a+b}$, we summarize as follows.

\begin{proposition}\label{prop:lower-disj-bicliques}
    There exist universal constants $a_0$ and $r_0$ with the following property.
    Let $F$ be a connected bipartite graph with bipartition sizes $a$ and $b$, where $a \geq a_0$ and $b \geq r_0 a$. Then there exists a constant $c_F > 1$, depending only on $F$, such that for every $\ve > 0$, there are infinitely many pairs $(n,m)$ satisfying
    \begin{equation*}
        N_F(n, m) \geq c_F^n \cdot \left[(1 - \ve) \left( \frac{|V(F)|}{|\mathrm{Aut}(F)|} \right)^{\frac{1}{|V(F)| - 1}} \frac{2m}{e n} \right]^{\left(1 - \frac{1}{|V(F)|}\right) n}.
    \end{equation*}
\end{proposition}


\section{Kahn--Lov\'{a}sz theorem for $F$-factors}\label{sec:F-factors}

The purpose of this section is to prove generalizations of the Kahn--Lov\'{a}sz theorem to Hamiltonian graphs (Theorem~\ref{thm:main-factors}) and to general connected graphs (Theorem~\ref{thm:general-factors}).
To this end, we begin with the following reduction lemma, which allows us to reduce to the cases of cycle factors and tree factors, respectively.

\begin{lemma}\label{lem:aut-double-counting}
    Let $F, H$, and $G$ be graphs such that $H$ is a spanning subgraph of $F$. Then we have 
    $$
        N(F; G) \leq \frac{|\Aut(H)|}{|\Aut(F)|} \cdot N(H; G).
    $$
\end{lemma}

\begin{proof}[Proof of Lemma~\ref{lem:aut-double-counting}]
    Denote $\mathrm{ext}(H; F)$ by the number of distinct ways to extend a given $H$ to $F$.

    \begin{claim}\label{clm:identity}
        $|\Aut(H)| \cdot N(H; F) = |\Aut(F)| \cdot \mathrm{ext}(H; F)$.
    \end{claim}

    \begin{claimproof}[Proof of Claim~\ref{clm:identity}]
        Let $n = v(F) = v(H)$. Consider the collection $\calC$ of pairs $(F', H')$ such that $H' \subseteq F' \subseteq K_n$ such that $F'$ and $H'$ are isomorphic to $F$ and $H$, respectively. Then the number of copies of $F$ in $K_n$ is $\frac{n!}{|\Aut(F)|}$, and by the definition, each such copy contains exactly $N(H; F)$ copies of $H$. Hence, we have
        \begin{equation}\label{eq:identity-1}
            |\calC| = \frac{n!}{|\Aut(F)|}N(H; F).
        \end{equation}
        On the other hand, the number of copies of $H$ in $K_n$ is $\frac{n!}{|\Aut(H)|}$. Also, by its definition, the number of distinct ways to extend each of such $H$ copies is exactly $\mathrm{ext}(H; F)$. Thus we have
        \begin{equation}\label{eq:identity-2}
            |\calC| = \frac{n!}{|\Aut(H)|} \cdot \mathrm{ext}(H; F).
        \end{equation}
        From \eqref{eq:identity-1} and \eqref{eq:identity-2}, we obtain the desired identity.
    \end{claimproof}

    Let $\calF$ and $\calH$ be the set of copies of $F$ and $H$ in $G$, respectively. Observe that $|\calF| = N(F; G)$ and $|\calH| = N(H; G)$. We now consider an auxiliary bipartite graph $M$ on the bipartition $\calF \cup \calH$ such that $(F', H') \in E(M)$ if and only if $H' \subseteq F'$. Then for each $H' \in \calH$, the degree of $H'$ in $M$ is at most $\mathrm{ext}(H; F)$, but for all $F' \in \calF$, its degree in $M$ is $N(H; F)$. This implies
    $
        e(M) = \sum_{F'\in \calF} N(H; F) \leq \sum_{H' \in \calH} \mathrm{ext}(H; F).
    $
    Hence, we deduce that
    \begin{equation}\label{eq:ext-inequality}
        N(F; G) \leq \frac{\mathrm{ext}(H; F)}{N(H; F)} \cdot N(H; G).
    \end{equation}
    From Claim~\ref{clm:identity}, we have the identity
    $
        \mathrm{ext}(H; F)/N(H; F) = |\Aut(H)|/|\Aut(F)|.
    $
    Together with \eqref{eq:ext-inequality}, we obtain
    $$
        N(F; G) \leq \frac{|\Aut(H)|}{|\Aut(F)|} \cdot N(H; G).
    $$
    This completes the proof.
    
\end{proof}

\subsection{Proof of Theorem~\ref{thm:main-factors}}\label{subsec:proof-main-factor}

We prove Theorem~\ref{thm:main-factors} first. Together with Lemma~\ref{lem:aut-double-counting}, it suffices to show the cycle factor case. For later purposes, we state it not in terms of $N_{\mathrm{factor}}(C_{\ell}; G)$ but in terms of the number of embeddings.

\begin{theorem}\label{thm:cycle-factor}
    For all integer $\ell \geq 2$ and real number $\ve > 0$, there exists a constant $C = C(\ell, \ve) > 0$ such that 
    \begin{equation*}
        \bigg|\mathrm{Emb}\left(\frac{n}{\ell}\cdot C_{\ell}; G\right)\bigg| \leq \left( \frac{n}{\ell} \right)! \cdot \left( (1  + \ve) \ell \right)^{\frac{n}{\ell}} \cdot \left( \prod_{v\in V(G)} \frac{d_G(v) + C}{e} \right)^{1 - \frac{1}{\ell}}
    \end{equation*}
    holds for all $n$-vertex graph $G$ where $n$ is divisible by $\ell$.
\end{theorem}

We note that $C_2$ is an edge. Keeping Theorem~\ref{thm:cycle-factor} in mind, the proof of Theorem~\ref{thm:main-factors} follows. 

\begin{proof}[Proof of Theorem~\ref{thm:main-factors}]
    Let $v(F) = \ell \geq 2$ and $v(G) = n$. If $n$ is not divisible by $\ell$, then obviously $N_{\mathrm{factor}}(F; G) = 0$. Hence, we may assume that $n$ is divisible by $\ell$ and let $m \defeq \frac{n}{\ell}$.
    Since $F$ is Hamiltonian, it contains $C_{\ell}$ as a subgraph. Then by Theorem~\ref{thm:cycle-factor}, for a given $\ve > 0$, there exists a constant $C$ that depends only on $\ve$ and $\ell$ such that
    \begin{equation}\label{eq:cycle-factor-eq}
        |\mathrm{Emb}\left(m \cdot C_{\ell}; G\right)| \leq m! \cdot \left( (1  + \ve) \ell \right)^{\frac{n}{\ell}} \cdot \left( \prod_{v\in V(G)} \frac{d_G(v) + C}{e} \right)^{1 - \frac{1}{\ell}}.
    \end{equation}
    We note that \eqref{eq:embedding-automorphism} implies 
    $
        N(m\cdot C_{\ell}; G) = |\mathrm{Emb}(m \cdot C_{\ell}; G)|/|\Aut(m \cdot C_{\ell})|
    $,
    so by Lemma~\ref{lem:aut-double-counting}, we have
    \begin{equation}\label{eq:factor-emb}
        N_{\mathrm{factor}}(F; G) = N(m \cdot F; G) \leq \frac{|\mathrm{Emb}(m \cdot C_{\ell}; G)|}{|\Aut(m \cdot F)|} = \frac{1}{m!} \cdot \left(\frac{1}{|\Aut(F)|}\right)^{\frac{n}{\ell}} \cdot |\mathrm{Emb}(m \cdot C_{\ell}; G)|.
    \end{equation}
    By combining \eqref{eq:cycle-factor-eq} and \eqref{eq:factor-emb}, we obtain
    $$
        N_{\mathrm{factor}}(F; G) \leq \left( (1  + \ve) \frac{\ell}{|\Aut(F)|} \right)^{\frac{n}{\ell}} \cdot \left( \prod_{v\in V(G)} \frac{d_G(v) + C}{e} \right)^{1 - \frac{1}{\ell}}.
    $$
    This completes the proof.
\end{proof}

To prove Theorem~\ref{thm:cycle-factor}, we consider pairs of partitions and digraphs as described in Section~\ref{subsec:overview}. For an integer $\ell \geq 2$, let $\ori{C_{\ell}}$ denote the labelled oriented cycle of length $\ell$ on the vertex set $\mathbb{Z}_{\ell}$ whose arcs are in the form of $i \to (i+1)$. Note that $\ori{C_{2}}$ is a $(\mathbb{Z}_2)$-labelled digon.

\begin{lemma}\label{lem:partition-cycle-sum}
    Let $n, \ell \geq 2$ be integers where $n$ is divisible by $\ell$. Then for all real number $\ve > 0$ and $C' > 0$, there exists $C > 0$ only depending on $\ve, \ell$, and $C'$ such that for all $n$-vertex graph $G$, it holds that
    \begin{equation}\label{eq:key-inequality}
        \sum_{\calP: \text{ $\ell$-partition}} \prod_{v\in V(G)} \bigg[ d^+_G\left(v; \left(\calP, \overrightarrow{C_{\ell}} \right) \right) + C' \bigg]
        \leq (1 + \ve)^n \prod_{v\in V(G)} (d_G(v) + C).
    \end{equation}
\end{lemma}

Before proving Lemma~\ref{lem:partition-cycle-sum}, we establish the following weighted counting lemma for spanning subdigraphs of out-degree one.

\begin{lemma}\label{lem:exponential-count}
    Let $\ell, C' \geq 0$ and $\ve > 0$ be real numbers. Then there exists $C = C(\ell, C', \ve) > 0$ such that the following holds. Let $H$ be an $n$-vertex digraph in which every vertex has at most $C'$ directed loops, and let $\calD$ be the set of spanning subdigraphs of $H$ in which every vertex has out-degree $1$. Then we have
    \begin{equation*}
        \sum_{D\in \calD} \ell^{|\mathrm{Comp}(D)|} \leq (1 + \ve)^n \prod_{v\in V(H)} (d^+_H(v) + C).
    \end{equation*}
\end{lemma}

\begin{proof}[Proof of Lemma~\ref{lem:exponential-count}]
    By monotonicity, we may assume that $\ell \geq 2$ and $\ve < 1/2$. We choose $C$ at the end of the proof.
    Observe that $|\calD| = \prod_{v\in V(H)}d^+_H(v)$ as $\calD$ is the set of out-degree $1$-regular spanning subgraphs of $H$, and for each $v\in V(H)$, the number of ways to choose its unique out-neighbor is exactly $d^+_H(v)$.
    Set $T \defeq \log_{\ell} \left( (1+\ve)^n / 2\right)$ and let $\calD_1 \defeq \{D\in \calD: |\mathrm{Comp}(D)|\leq T\}$ and $\calD_2 \defeq \calD \setminus \calD_1$. 
    Then we have
    \begin{equation}\label{eq:D1}
        \sum_{D\in \calD_1} \ell^{|\mathrm{Comp}(D)|} \leq \ell^T |\calD_1| \leq \frac{1}{2}\cdot (1 + \ve)^n |\calD| \leq \frac{1}{2}\cdot (1 + \ve)^n \prod_{v\in V(H)} (d^+_H(v) + C).
    \end{equation}

    We now consider $\calD_2$. Obviously, we have
    \begin{equation}\label{eq:trivial-D2}
        \sum_{D\in \calD_2} \ell^{|\mathrm{Comp}(D)|} \leq \ell^n |\calD_2|,
    \end{equation}
    so it remains to bound the size of $\calD_2$.

    For every digraph $D\in \calD_2$, at least $\lceil |\mathrm{Comp}(D)|/2 \rceil \geq \lceil T/2 \rceil$ components of $D$ have size at most $2n / T$ by Markov's inequality. We enumerate the set $\mathrm{Comp}(D) = \{D_1, \dots, D_t\}$ with $|V(D_1)| \leq |V(D_2)| \leq \cdots \leq |V(D_t)|$. We note that $t \geq T$ and $|V(D_i)| \leq 2n/T$ for all $i \leq \lceil T/2 \rceil$.
    Also, since each component of $D$ is a connected out-degree-$1$ regular digraph, it contains a unique directed cycle as a subdigraph. Note that we also consider a loop to be a directed cycle. Hence, for each $i\in [t]$, the digraph $D_i$ contains a vertex $v_i$ that lies in a directed cycle. This means $D_i$ remains connected even after removing the unique directed edge from $v_i$.
    Denote $R(D) \subseteq V(D)$ by the set $\{v_1, \dots, v_{\lceil T/2 \rceil}\}$. 
    
    Then it holds that
    \begin{equation}\label{eq:D2-roots}
        |\calD_2| \leq \sum_{R\in \binom{V(H)}{\lceil T/2 \rceil}} |\{D\in \calD_2: R(D) = R\}|.
    \end{equation}

    Denote $\calD_R$ by the set $\{D\in \calD_2: R(D) = R\}$.
    We now fix $R\subseteq V(H)$ of size $\lceil T/2 \rceil$. To bound the quantity $|\calD_R|$, we first assign the unique out-neighbor of each vertex $v\in V(H)\setminus R$. Let $\calD'_R$ be the set of subdigraphs $D'$ of $H$ such that $d^+_{D'}(v) = 0$ if $v\in R$ and $d^+_{D'}(v) = 1$ otherwise. Since each vertex $v\in V(H)\setminus R$ has at most $d_H^+(v)$ possible out-neighbors, we have
    \begin{equation}\label{eq:D'R}
        |\calD'_R| \leq \prod_{v\in V(H)\setminus R} d^+_H(v).
    \end{equation}
    
    Let $D'$ be an arbitrary member of $\calD'_R$. We claim that the number of distinct ways to extend $D'$ to a member of $\calD_R$ is bounded. 
    Observe that for every $D\in \calD_R$, each $v\in R$ has its unique neighbor in the connected component of $D$ containing $v$, whose size is at most $2n/T$. Thus, for each $v\in R$, there are at most $2n/T + C'$ choices for the directed edge from $v$ when extending $D'$ to a member of $\calD_R$, since there are at most $C'$ directed loops at $v$. Hence, the number of distinct extensions of $D'$ to a member of $\calD_R$ is at most $(2n/T + C')^{\lceil T/2 \rceil}$. Together with \eqref{eq:D'R}, we obtain
    \allowdisplaybreaks
    \begin{align}
        |\calD_R| &\leq (2n/T + C')^{\lceil T/2\rceil} \cdot |\calD'_R| \nonumber \\
        &\leq (2n/T + C')^{\lceil T/2\rceil} \cdot \prod_{v\in V(H)\setminus R} d^+_H(v) \nonumber \\
        &\leq \prod_{r\in R} \frac{2n/T + C'}{d^+_H(r) + C} \cdot \prod_{v\in V(H)} (d^+_H(v) + C) \nonumber \\
        &\leq \left(\frac{2n/T + C'}{C}\right)^{T/2}  \cdot \prod_{v\in V(H)} (d^+_H(v) + C). \label{eq:DR-last}
    \end{align}

    Let $\alpha \defeq \frac{2 \log (1 + \ve)}{\log \ell} > 0$. By choosing $C$ large enough, we may assume that $n \geq \frac{2 \log \ell}{\log (1 + \ve)}$. Then we have
    $$
        T = \log_{\ell} \left( \frac{(1 + \ve)^n}{2}\right) \geq \alpha n.
    $$
    Then by \eqref{eq:trivial-D2}, \eqref{eq:D2-roots}, \eqref{eq:DR-last}, it holds that
    \allowdisplaybreaks
    \begin{align}
        \sum_{D\in \calD_2} \ell^{|\mathrm{Comp}(D)|} &\leq \ell^n |\calD_2| \nonumber \\
        &\leq \ell^n \sum_{R\in \binom{V(H)}{\lceil T/2 \rceil}} |\calD_R| \nonumber \\
        &\leq \ell^n \cdot 2^n \cdot \left(\frac{2n/T + C'}{C}\right)^{T/2}  \cdot \prod_{v\in V(H)} (d^+_H(v) + C) \nonumber \\
        &\leq \left[2\ell \cdot \left( \frac{2\alpha + C'}{C} \right)^{\alpha / 2}  \right]^n \cdot \prod_{v\in V(H)} (d^+_H(v) + C). \label{weighted-D2-last}
    \end{align}
    Since $\alpha$ is a positive constant that only depends on $\ell, \ve, C'$, there exists large enough constant $C_0 > 0$ such that 
    $\left[2\ell \cdot \left( \frac{2\alpha + C'}{C} \right)^{\alpha / 2}  \right]^n \leq \frac{1}{2} \cdot (1 + \ve)^n$ holds. We now take $C = C_0$. Then by \eqref{weighted-D2-last}, we finally deduce that
    \begin{equation}\label{eq:D2-final}
        \sum_{D\in \calD_2} \ell^{|\mathrm{Comp}(D)|} \leq \frac{1}{2}\cdot (1 + \ve)^n \cdot \prod_{v\in V(H)} (d^+_H(v) + C).
    \end{equation}
    
    By combining \eqref{eq:D1} and \eqref{eq:D2-final}, we obtain the desired inequality. This completes the proof.
\end{proof}

We now prove Lemma~\ref{lem:partition-cycle-sum}.

\begin{proof}[Proof of Lemma~\ref{lem:partition-cycle-sum}]
    We may assume that $C'$ is a positive integer. Let $\dir{G'}$ be a digraph obtained from $G$ by replacing each edge of $G$ by a digon and adding exactly $C'$ directed loops to every vertex. Denote $\calD$ the set of out-degree $1$-regular spanning subdigraphs of $\dir{G'}$. For a digraph $D$, we say a function $f: V(D)\to \mathbb{Z}_{\ell}$ as \emph{cyclic labelling} of $D$ if $f(v) = f(u) + 1$ whenever $u\neq v$ and $\ori{uv}\in E(D)$.

    \begin{claim}\label{clm:cyclic-labelling}
        Let $D$ be a digraph. Then the number of cyclic labellings of $D$ is at most $\ell^{|\mathrm{Comp}(D)|}$.
    \end{claim}

    \begin{claimproof}[Proof of Claim~\ref{clm:cyclic-labelling}]
        Let $D'$ be one of the components of $D$. It suffices to show that the number of cyclic labellings of $D'$ is at most $\ell$.
        Choose an arbitrary vertex $v\in V(D')$. Let $f'$ be a cyclic labelling of $D'$ with $f'(v) = i$ for some value $i\in \mathbb{Z}_{\ell}$. Since the underlying graph of $D'$ is connected, for all $u\in V(D') \setminus \{v\}$, there exists an oriented path $P$ from $v$ to $u$ in $D'$. Assume $P$ has $x$ and $y$ directed forward and directed backward to $u$, respectively. Then by the definition of cyclic labelling, we have $f'(u) \equiv i + x - y \pmod{\ell}$. Hence, $f'$ is uniquely determined from the value of $f'(v)$. As the number of possible choices for $f'(v)$ is $\ell$, the number of cyclic labellings of $D'$ is at most $\ell$. This completes the proof.
    \end{claimproof}

    We now show that a double-counting argument gives that the left-hand side of \eqref{eq:key-inequality} is at most the sum, over all $D \in \calD$, of the number of cyclic labellings of $D$. Together with Claim~\ref{clm:cyclic-labelling}, we claim the following.

    \begin{claim}\label{clm:double-counting-components}
        $
            \text{(LHS) of \eqref{eq:key-inequality}} \leq \sum_{D\in \calD} \ell^{|\mathrm{Comp}(D)|}
        $
    \end{claim}

    \begin{claimproof}[Proof of Claim~\ref{clm:double-counting-components}]
        Let $\calQ$ be the set of pairs $(D, \calP)$ such that $D \in \calD$ and $\calP$ is an $\ell$-partition of $V(G)$ where $\calP = (V_0, \dots, V_{\ell-1})$ that satisfies the following property. For every non-loop arc $\ori{uv}$, there exists $i\in \mathbb{Z}_{\ell}$ such that $u\in V_i$ and $v\in V_{i+1}$.
        For a fixed $\ell$-partition $\calP$, the value $\prod_{v\in V(G)} \bigg[ d^+_G\left(v; \left(\calP, \ori{C_{\ell}} \right) \right) + C' \bigg]$ is equal to the number of digraphs $D\in \calD$ such that $(D, \calP)\in \calQ$. Thus, the left-hand side of \eqref{eq:key-inequality} equals $|\calQ|$. By double counting, it suffices to show that for each $D\in \calD$, the number of $ \ell$-partitions $\calP$ such that $(D, \calP) \in \calQ$ is at most $\ell^{|\mathrm{Comp}(D)|}$ to conclude the proof. 
        For a fixed $D\in \calD$, an $\ell$-partition $\calP = (V_0, \dots, V_{\ell-1})$ such that $(D, \calP)\in \calQ$ corresponds to a unique cyclic labelling $f$ of $D$ by defining $f^{-1}(i) := V_i$ for each $i\in \mathbb{Z}_{\ell}$. Thus, the number of such $\calP$ is at most the number of cyclic labellings of $D$, which is at most $\ell^{|\mathrm{Comp}(D)|}$ given by Claim~\ref{clm:cyclic-labelling}. This completes the proof.
    \end{claimproof}

    Since $\dir{G}$ is a digraph with the number of loops for each vertex being at most $C'$, by Lemma~\ref{lem:exponential-count}, there exists $C > 0$ that only depends on $\ell, \ve, C'$ such that
    $$
        \sum_{D\in \calD} \ell^{|\mathrm{Comp}(D)|} \leq (1 + \ve)^n \prod_{v\in V(G)} (d_G(v) + C).
    $$ 
    Together with Claim~\ref{clm:double-counting-components}, we obtain the desired inequality. This completes the proof.
\end{proof}

We are now ready to prove Theorem~\ref{thm:cycle-factor}. A key ingredient in the proof is H\"{o}lder's inequality.

\begin{proof}[Proof of Theorem~\ref{thm:cycle-factor}]
    We fix parameters as
    $$
        0 < \frac{1}{C} \ll \frac{1}{C'} \ll \ve' \ll \ve, \frac{1}{\ell} < 1.
    $$
    
    Let $m\defeq n/\ell$.
    To count the number of embeddings of $C_{\ell}$-factors, namely $m \cdot C_{\ell}$, we fix an enumeration of cycles of length $\ell$ in the graph $m \cdot C_{\ell}$ as $C^{(1)}_{\ell}, \dots, C^{(m)}_{\ell}$. Assume that for each cycle of length $\ell$, the vertices are cyclically labelled with $\mathbb{Z}_{\ell}$ and we denote $c(v)$ as such label for each $v\in V(m \cdot C_{\ell})$. Then we define the labelling function $\iota: V(m \cdot C_{\ell}) \to \mathbb{Z}_{\ell} \times [m]$ such that $\iota(v) = (c(v), t)$ if $v\in C^{(t)}_{\ell}$. 

    As we aim to bound the number of embeddings of a $C_{\ell}$-factor, we first assign labels from $\mathbb{Z}_{\ell}$ to all the vertices of $G$. To this end, for an arbitrary $\ell$-equipartition $\calP = (V_0, \dots, V_{\ell-1})$ of $V(G)$, let $\mathrm{Emb}_{\calP}(m \cdot C_{\ell}; G)$ the set of embeddings $\phi \in \mathrm{Emb}(m \cdot C_{\ell}; G)$ such that $\phi(v) \in V_i$ if and only if $(\iota(v))_1 = i$ for each $v\in V(m \cdot C_{\ell})$ and $i\in \mathbb{Z}_{\ell}$.
    Then, we have
    \begin{equation*}
        \lvert \mathrm{Emb}(m \cdot C_{\ell}; G) \rvert = \sum_{\calP: \text{ $\ell$-equipartition}} \lvert \mathrm{Emb}_{\calP}(m \cdot C_{\ell}; G) \rvert.
    \end{equation*}

    Recall that $\ori{C_\ell}$ denotes the directed cycle on the vertex set $\mathbb{Z}_{\ell}$ with edges $i\to (i+1)$ for all $i\in \mathbb{Z}_{\ell}$. 

    \begin{claim}\label{clm:matching}
        For a given $\ell$-equipartition $\calP = (V_0, \dots, V_{\ell - 1})$ and each $i\in \mathbb{Z}_{\ell}$, it holds that
        \begin{equation}\label{eq:upperbound-from-dir-path}
            \lvert \mathrm{Emb}_{\calP}(m \cdot C_{\ell}; G) \rvert \leq m! \cdot \left( \frac{1 + \ve'}{e} \right)^{(\ell - 1)m} \cdot \prod_{v\in V(G)\setminus V_{i}} \bigg[ d_G^+\left(v; \left(\calP, \ori{C_{\ell}}\right)\right) + C' \bigg].
        \end{equation}
    \end{claim}

    \begin{claimproof}[Proof of Claim~\ref{clm:matching}]
        Fix an $\ell$-equipartition $\calP = (V_0, \dots, V_{\ell - 1})$ and $i\in \mathbb{Z}_{\ell}$. For each $j\in \mathbb{Z}_{\ell} \setminus \{i\}$, let $M_j$ be the set of perfect matchings in the bipartite graph $G[V_j, V_{j+1}]$, which is a bipartite subgraph induced by $V_j$ and $V_{j+1}$. 
        We observe that for each $j\in \mathbb{Z}_{\ell} \setminus \{i\}$, an arbitrary choice of $\psi_j\in M_j$, the graph $\bigcup_{j\in \mathbb{Z}_{\ell} \setminus \{i\}} \psi_j$ forms a $P_{\ell}$-factor of $G$, here $P_{\ell}$ is a path of length $\ell - 1$. Also, each image of $\phi \in \mathrm{Emb}_{\calP}(m\cdot C_{\ell}; G)$ projected between $V_j$ and $V_{j+1}$ is a perfect matching. Also, there is at most one way to extend such an $P_{\ell}$-factor to a $C_{\ell}$-factor, which is adding edges between the two endpoints of each path if it exists. By considering the orderings of each path in a $P_{\ell}$-factor, we need to multiply $m!$ since such a $P_{\ell}$-factor consists of $m$ paths. Thus, we have
        \begin{equation}\label{eq:prod-matching}
            \lvert \mathrm{Emb}_{\calP}(m \cdot C_{\ell}; G) \rvert \leq m! \cdot \prod_{j\in \mathbb{Z}_{\ell}\setminus \{i\}} |M_j|.
        \end{equation}

        Let $H_j \defeq G[V_j, V_{j+1}]$. Then Br\'{e}gman--Minc inequality (Theorem~\ref{thm:bregman-minc}) implies that
        \begin{equation*}
            |M_j| \leq \prod_{v\in V_j} \left( d_{H_j}(v)! \right)^{\frac{1}{d_{H_j}(v)}}.
        \end{equation*}
        Together with Proposition~\ref{prop:factorial-linear}, we have
        \begin{equation}\label{eq:linear-bregman-minc}
            |M_j| \leq \left( \frac{1 + \ve'}{e} \right)^m \cdot \prod_{v\in V_j} \left(d_{H_j}(v) + C'\right).
        \end{equation}
        Observe that for each $j\neq i$ and $v\in V_j$, the equality $d_{H_j}(v) =  d_G^+\left(v; \left(\calP, \ori{C_{\ell}}\right)\right)$ holds. Hence, from \eqref{eq:linear-bregman-minc}, we have
        \begin{equation*}
            |M_j| \leq \left( \frac{1 + \ve'}{e} \right)^m \cdot \prod_{v\in V_j} \bigg[ d_G^+\left(v; \left(\calP, \ori{C_{\ell}}\right)\right) + C' \bigg].
        \end{equation*}
        This, together with \eqref{eq:prod-matching}, yields the desired inequality. This completes the proof.
    \end{claimproof}

    By applying \eqref{eq:upperbound-from-dir-path} for all $i\in \mathbb{Z}_{\ell}$ and taking a geometric mean, we have
    \allowdisplaybreaks
    \begin{align*}
        \lvert \mathrm{Emb}_{\calP}(m \cdot C_{\ell}; G) \rvert &\leq m! \cdot \left( \frac{1 + \ve'}{e} \right)^{(\ell - 1)m} \cdot \bigg[ \prod_{i\in \mathbb{Z}_{\ell}}\prod_{v\in V(G)\setminus V_{i}} \bigg[ d_G^+\left(v; \left(\calP, \ori{C_{\ell}}\right)\right) + C' \bigg] \bigg]^{1/\ell}\\
        &= m! \cdot \left( \frac{1 + \ve'}{e} \right)^{(\ell - 1)m} \cdot  \prod_{v\in V(G)} \bigg[ d_G^+\left(v; \left(\calP, \ori{C_{\ell}}\right)\right) + C' \bigg]^{(\ell-1)/\ell}
    \end{align*}
    The last equality holds since for each $v\in V(G)$, the term $\big[ d_G^+\left(v; \left(\calP, \ori{C_{\ell}}\right)\right) + C' \big]$ appears exactly $(\ell - 1)$ times in the right-hand side of the first inequality. Then by H\"{o}lder's inequality, we have
    \allowdisplaybreaks
    \begin{align}
        \lvert \mathrm{Emb}(m \cdot C_{\ell}; G) \rvert &= \sum_{\calP: \text{ $\ell$-equipartition}} \lvert \mathrm{Emb}_{\calP}(m \cdot C_{\ell}; G) \rvert \nonumber \\
        &\leq  m! \cdot \left( \frac{1 + \ve'}{e} \right)^{(\ell - 1)m} \cdot \sum_{\calP: \text{ $\ell$-equipartition}} \prod_{v\in V(G)} \bigg[ d_G^+\left(v; \left(\calP, \ori{C_{\ell}}\right)\right) + C' \bigg]^{(\ell-1)/\ell} \nonumber \\
        &\leq m! \cdot \left( \frac{1 + \ve'}{e} \right)^{(\ell - 1)m} \cdot \left( \sum_{\calP: \text{ $\ell$-equipartition}} 1^{\ell} \right)^{1/\ell} \nonumber \\
        &\qquad \qquad \cdot \bigg[ \sum_{\calP: \text{ $\ell$-equipartition}} \prod_{v\in V(G)} \bigg[ d_G^+\left(v; \left(\calP, \ori{C_{\ell}}\right)\right) + C' \bigg]\bigg]^{(\ell-1)/\ell}. \label{eq:after-Holder} 
    \end{align}
    
    We note that the number of $\ell$-equipartitions of a set of size $n$ is at most $\ell^n$. Thus,
    \begin{equation}\label{eq:equi-partition-upper}
        \left( \sum_{\calP: \text{ $\ell$-equipartition}} 1^{\ell} \right)^{1/\ell} \leq \ell^{n/\ell}.
    \end{equation}
    Also, by Lemma~\ref{lem:partition-cycle-sum}, we have
    \begin{equation}\label{eq:key-2}
        \sum_{\calP: \text{ $\ell$-equipartition}} \prod_{v\in V(G)} \bigg[ d^+_G\left(v; \left(\calP, \ori{C_{\ell}} \right) \right) + C' \bigg]
        \leq (1 + \ve')^n \prod_{v\in V(G)} (d_G(v) + C).
    \end{equation}
    By combining \eqref{eq:after-Holder}, \eqref{eq:equi-partition-upper}, and \eqref{eq:key-2}, we obtain
    \allowdisplaybreaks
    \begin{align*}
        \lvert \mathrm{Emb}(m \cdot C_{\ell}; G) \rvert &\leq m! \cdot \left( \frac{1 + \ve'}{e} \right)^{(\ell - 1)m}  \cdot \ell^{n/\ell} \cdot (1 + \ve')^{(\ell - 1)m} \cdot \left(\prod_{v\in V(G)} (d_G(v) + C) \right)^{1 - \frac{1}{\ell}}\\
        &\leq \left( \frac{n}{\ell} \right)! \cdot \left( (1  + \ve) \ell \right)^{\frac{n}{\ell}} \cdot \left( \prod_{v\in V(G)} \frac{d_G(v) + C}{e} \right)^{1 - \frac{1}{\ell}}.
    \end{align*}
    This completes the proof.
\end{proof}

\subsection{General connected graphs}\label{subsec:general}

The purpose of this section is to establish a Kahn--Lov\'{a}sz-type theorem for general connected graphs $F$, stated below.

\begin{theorem}\label{thm:general-factors}
    Let $F$ be a connected graph with a spanning tree whose bipartition classes have sizes $a$ and $b$. Then for every $\ve > 0$, there exists a constant $C = C(a + b, \ve) > 0$ such that the following holds. For all $n$-vertex graph $G$, we have
    $$
        N_{\mathrm{factor}}(F; G) \leq \left( (1  + \ve) \frac{|V(F)|}{|\mathrm{Aut}(F)|} \right)^{\frac{n}{|V(F)|}} \cdot \left( a^{a}b^{b} \right)^{\frac{|V(F)|-1}{|V(F)|^2}n} \cdot \left( \prod_{v\in V(G)} \frac{d_G(v) + C}{e} \right)^{1 - \frac{1}{|V(F)|}}.
    $$
\end{theorem}

Since every connected graph contains a spanning tree, the following theorem implies Theorem~\ref{thm:main-factors}.

\begin{theorem}\label{thm:tree-factor}
    For every tree $T$ whose bipartition classes have sizes $a$ and $b$, and every real number $\ve>0$, there exists a constant $C = C(a + b, \ve) > 0$ such that 
    \begin{equation*}
        \bigg|\mathrm{Emb}\left(\frac{n}{a + b}\cdot T; G\right)\bigg| \leq \left( \frac{n}{a+b} \right)! \cdot ((1 + \ve)(a+b))^{\frac{n}{a+b}} \cdot \left( a^{a}b^{b} \right)^{\frac{(a+b-1)}{(a+b)^2}n} \cdot \left( \prod_{v\in V(G)} \frac{d_G(v) + C}{e} \right)^{1 - \frac{1}{a+b}}
    \end{equation*}
    holds for all $n$-vertex graph $G$ where $n$ is divisible by $a + b$.
\end{theorem}

Theorem~\ref{thm:general-factors} follows from Theorem~\ref{thm:tree-factor} and Lemma~\ref{lem:aut-double-counting} exactly as Theorem~\ref{thm:main-factors} follows from Theorem~\ref{thm:cycle-factor}; we omit the details. The proof of Theorem~\ref{thm:tree-factor} follows the same strategy as that of Theorem~\ref{thm:cycle-factor}, with suitable modifications. We start with the following lemma.

\begin{lemma}\label{lem:partition-tree-sum}
    Let $n, a, b \geq 1$ be integers where $n$ is divisible by $a + b$. Let $T$ be a labelled tree on the vertex set $\mathbb{Z}_{(a+b)}$, whose unique bipartition has size $a$ and $b$. Denote $\dir{T}$ by the digraph obtained from $T$ by replacing every edge of $T$ with a digon.
    Then for all real number $\ve > 0$ and $C' > 0$, there exists $C > 0$ only depending on $\ve, a+b$, and $C'$ such that for all $n$-vertex graph $G$, it holds that
    \begin{equation}\label{eq:key-inequality-tree}
        \sum_{\calP: \text{ $(a+b)$-equipartition}} \prod_{v\in V(G)} \bigg[ d^+_G\left(v; \left(\calP, \dir{T} \right) \right) + C' \bigg]
        \leq \left((1 + \ve) a^{\frac{a}{a+b}} b^{\frac{b}{a+b}}\right)^n \cdot \prod_{v\in V(G)} (d_G(v) + C).
    \end{equation}
\end{lemma}

\begin{proof}[Proof of Lemma~\ref{lem:partition-tree-sum}]
    Since $T$ is a bipartite graph, we may assume that its bipartition is $(\{0, \dots, a-1\}, \{a, \dots, a+b-1 \})$. For a given $(a+b)$-equipartition $\calP = (V_0, \dots, V_{a+b-1})$, We denote by $B(\calP)$ the $2$-partition $(U_0, U_1)$ such that $U_0 = \bigcup_{0\leq i \leq a-1} V_i$ and $U_1 = V(G)\setminus U_0$.
    We observe that by its definition, we have
    \begin{equation}\label{eq:P-BP-outdegree}
        d^+_G\left(v; \left(\calP, \dir{T} \right) \right) \leq d^+_G\left(v; \left(B(\calP), \ori{C_2} \right) \right)
    \end{equation}
    for all $(a+b)$-partitions $\calP$ and vertices $v\in V(G)$.

    \begin{claim}\label{clm:number-of-B(P)}
        For each $2$-partition $\calQ$, the number of $t$-equipartitions $\calP$ with $\calQ = B(\calP)$ is at most
        $$
            \left( a^{\frac{a}{a+b}} b^{\frac{b}{a+b}}\right)^n.
        $$
    \end{claim}

    \begin{claimproof}[Proof of Claim~\ref{clm:number-of-B(P)}]
        Let $\calQ = (U_0, U_1)$. If $|U_0| \neq an/(a+b)$, then there is no $(a+b)$-equipartition $\calP$ with $\calQ = B(\calP)$. Thus, we may assume that $|U_0| = an/(a+b)$ and $|U_1| = bn/(a+b)$. Then the number of such $(a+b)$-equipartitions is bounded by the product of the number of $a$-equipartitions of $U_0$ and $b$-equipartitions of $U_1$, which is at most
        $$
            \left( a^{\frac{a}{a+b}} b^{\frac{b}{a+b}}\right)^n.
        $$
        This completes the proof.
    \end{claimproof}

    By \eqref{eq:P-BP-outdegree} and Claim~\ref{clm:number-of-B(P)}, we have
    \allowdisplaybreaks
    \begin{align*}
        \text{(LHS) of \eqref{eq:key-inequality-tree}} &\leq \left( a^{\frac{a}{a+b}} b^{\frac{b}{a+b}}\right)^n \cdot \sum_{\calQ: \text{ $2$-partition}} \prod_{v\in V(G)} \bigg[ d^+_G\left(v; \left(\calQ, \ori{C_2} \right) \right) + C' \bigg]\\
        &\leq \left((1 + \ve) a^{\frac{a}{a+b}} b^{\frac{b}{a+b}}\right)^n \cdot \prod_{v\in V(G)} (d_G(v) + C),
    \end{align*}
    where the last inequality holds by Lemma~\ref{lem:partition-cycle-sum}. This completes the proof.
\end{proof}

We now prove Theorem~\ref{thm:tree-factor}.

\begin{proof}[Proof of Theorem~\ref{thm:tree-factor}]
    We fix parameters as
    $$
        0 < \frac{1}{C} \ll \frac{1}{C'} \ll \ve' \ll \ve, \frac{1}{a+b} < 1.
    $$
    
    Let $m\defeq n/(a + b)$.
    Similarly to the proof of Theorem~\ref{thm:cycle-factor}, we fix an enumeration of trees $T$ in the graph $m\cdot T$ as $T^{(1)}, \dots, T^{(m)}$. We now regard $T$ as a labelled tree on the vertex set $\mathbb{Z}_{t}$ and denote $c(v)$ as such a label for each $v\in V(m\cdot T)$.
    Define $\iota: V(m\cdot T)\to \mathbb{Z}_{(a+b)} \times [m]$ as $\iota(v) = (c(v), s)$ if $v\in T^{(s)}$.

    Similarly to the proof of Theorem~\ref{thm:cycle-factor}, for a given $(a+b)$-equipartition $\calP = (V_0, \dots, V_{a+b-1})$, we define $\mathrm{Emb}_{\calP}(m\cdot T; G)$ by the set of embeddings $\phi \in \mathrm{Emb}(m\cdot T; G)$ such that $\phi(v)\in V_i$ if and only if $(\iota(v))_1 = i$ for each $v\in V(m\cdot T)$ and $i\in \mathbb{Z}_{(a+b)}$. Then we have
    \begin{equation*}
        \lvert \mathrm{Emb}(m \cdot T; G) \rvert = \sum_{\calP: \text{ $(a+b)$-equipartition}} \lvert \mathrm{Emb}_{\calP}(m \cdot T; G) \rvert.
    \end{equation*}

    We denote by $\dir{T}$ the directed graph obtained from $T$ by replacing every edge with a digon.
    For each $i\in \mathbb{Z}_{(a + b)}$, we write $\ori{T_i}$ as an oriented graph whose underlying graph is $T$ in which every edge is directed towards $i$ along the unique path to $i$. We note that all the vertices of $\ori{T_i}$ except $i$ have out-degree one and $i$ has out-degree zero.

    \begin{claim}\label{clm:matching-tree}
        For a given $(a+b)$-equipartition $\calP = (V_0, \dots, V_{a+b-1})$ and each $i\in \mathbb{Z}_{(a+b)}$, it holds that
        \begin{equation}\label{eq:upperbound-from-rooted-tree}
            \lvert \mathrm{Emb}_{\calP}(m \cdot T; G) \rvert \leq m! \cdot \left( \frac{1 + \ve'}{e} \right)^{(a + b - 1)m} \cdot \prod_{v\in V(G)\setminus V_{i}} \bigg[ d_G^+\left(v; \left(\calP, \ori{T_i}\right)\right) + C' \bigg].
        \end{equation}
    \end{claim}

    \begin{claimproof}[Proof of Claim~\ref{clm:matching-tree}]
        For each $j\in \mathbb{Z}_{(a+b)} \setminus \{i\}$, denote $j'$ by the unique out-neighbor of $j$. Let $H_j$ be the bipartite graph $G[V_j, V_{j'}]$ and let $M_j$ be the set of perfect matchings of $H_j$. We note that for each $\phi \in \mathrm{Emb}_{\calP}(m\cdot T; G)$ and $j\in \mathbb{Z}_{(a+b)}\setminus \{i\}$, the projection of $\phi$ on to $H_j$ induces a perfect matching. Also, the union of perfect matchings chosen from $M_j$ forms an image $m\cdot T$ for some $\phi \in \mathrm{Emb}_{\calP}(m \cdot T; G)$. Thus, we have
        \begin{equation}\label{eq:prod-matching-tree}
            \lvert \mathrm{Emb}_{\calP}(m \cdot T; G) \rvert \leq m! \cdot \prod_{j\in \mathbb{Z}_{(a + b)}\setminus \{i\}} |M_j|.
        \end{equation}

        Similarly to the proof of Claim~\ref{clm:matching}, Br\'{e}gman--Minc inequality (Theorem~\ref{thm:bregman-minc}) and Proposition~\ref{prop:factorial-linear} implies that
        \begin{equation}\label{eq:linear-bregman-minc-tree}
            |M_j| \leq \left( \frac{1 + \ve'}{e} \right)^m \cdot \prod_{v\in V_j} \left(d_{H_j}(v) + C'\right).
        \end{equation}
        Since $d_{H_j}(v) = d_G^+\left(v; \left(\calP, \ori{T_i}\right)\right)$ for each $j\in \mathbb{Z}_{(a+b)}\setminus \{i\}$ and $v\in V_j$, the inequalities \eqref{eq:prod-matching-tree} and \eqref{eq:linear-bregman-minc-tree} implies that
        \allowdisplaybreaks
        \begin{align*}
            \lvert \mathrm{Emb}_{\calP}(m \cdot T; G) \rvert &\leq m! \cdot \prod_{j\in \mathbb{Z}_{(a + b)}\setminus \{i\}} |M_j| \\
            &\leq m! \cdot \left( \frac{1 + \ve'}{e} \right)^{(a + b - 1)m} \cdot \prod_{v\in V(G)\setminus V_{i}} \bigg[ d_G^+\left(v; \left(\calP, \ori{T_i}\right)\right) + C' \bigg].
        \end{align*}
        This completes the proof.
    \end{claimproof}

    Observe that for a given $(a+b)$-equipartition $\calP = (V_0, \dots, V_{a+b-1})$ and $i\in \mathbb{Z}_{(a+b)}$, it holds that 
    $$
        d_G^+\left(v; \left(\calP, \ori{T_i}\right)\right) \leq d_G^+\left(v; \left(\calP, \dir{T}\right)\right)
    $$ for all $v\in V(G)\setminus V_i$.
    Thus, \eqref{eq:upperbound-from-rooted-tree} implies that
    \begin{equation}\label{eq:upperbound-from-tree}
        \lvert \mathrm{Emb}_{\calP}(m \cdot T; G) \rvert \leq m! \cdot \left( \frac{1 + \ve'}{e} \right)^{(a + b - 1)m} \cdot \prod_{v\in V(G)\setminus V_{i}} \bigg[ d_G^+\left(v; \left(\calP, \dir{T}\right)\right) + C' \bigg].
    \end{equation}
    By applying \eqref{eq:upperbound-from-tree} for all $i\in \mathbb{Z}_{(a+b)}$ and taking a geometric mean, we have
    \allowdisplaybreaks
    \begin{align*}
        \lvert \mathrm{Emb}_{\calP}(m \cdot T; G) \rvert &\leq m! \cdot \left( \frac{1 + \ve'}{e} \right)^{(a + b - 1)m} \cdot \bigg[ \prod_{i\in \mathbb{Z}_{(a+b)}}\prod_{v\in V(G)\setminus V_{i}} \bigg[ d_G^+\left(v; \left(\calP, \dir{T}\right)\right) + C' \bigg] \bigg]^{1/(a+b)}\\
        &= m! \cdot \left( \frac{1 + \ve'}{e} \right)^{(a + b - 1)m} \cdot  \prod_{v\in V(G)} \bigg[ d_G^+\left(v; \left(\calP, \dir{T}\right)\right) + C' \bigg]^{(a + b - 1)/(a + b)}.
    \end{align*}
    The last equality holds since for each $v\in V(G)$, the term $\big[ d_G^+\left(v; \left(\calP, \dir{T}\right)\right) + C' \big]$ appears exactly $(a + b - 1)$ times in the right-hand side of the first inequality.
    We now apply H\"{o}lder's inequality similarly to the proof of Theorem~\ref{thm:cycle-factor}. Together with Lemma~\ref{lem:partition-tree-sum}, we have
    \allowdisplaybreaks
    \begin{align}
        \lvert \mathrm{Emb}(m \cdot T; G) \rvert &= \sum_{\calP: \text{ $(a + b)$-equipartition}} \lvert \mathrm{Emb}_{\calP}(m \cdot T; G) \rvert \nonumber \\
        &\leq  m! \cdot \left( \frac{1 + \ve'}{e} \right)^{(a + b - 1)m} \cdot \sum_{\calP: \text{ $(a + b)$-equipartition}} \prod_{v\in V(G)} \bigg[ d_G^+\left(v; \left(\calP, \dir{T}\right)\right) + C' \bigg]^{(a + b - 1)/(a + b)} \nonumber \\
        &\leq m! \cdot \left( \frac{1 + \ve'}{e} \right)^{(a + b - 1)m} \cdot \left( \sum_{\calP: \text{ $(a + b)$-equipartition}} 1^{(a+b)} \right)^{1/(a+b)} \nonumber \\
        &\qquad \cdot \bigg[ \sum_{\calP: \text{ $(a+b)$-equipartition}} \prod_{v\in V(G)} \bigg[ d_G^+\left(v; \left(\calP, \dir{T}\right)\right) + C' \bigg]\bigg]^{(a + b-1)/(a + b)} \label{eq:use-holder-tree} \\
        &\leq m! \cdot \left( \frac{1 + \ve'}{e} \right)^{(a + b - 1)m} \cdot (a+b)^m \nonumber \\
        &\qquad \qquad \cdot \left( \left((1 + \ve') a^{\frac{a}{a+b}} b^{\frac{b}{a+b}}\right)^n \cdot \prod_{v\in V(G)} (d_G(v) + C) \right)^{(a+b-1)/(a+b)} \label{eq:use-partition-tree-sum} \\
        &\leq \left( \frac{n}{a+b} \right)! \cdot ((1 + \ve)(a+b))^{\frac{n}{a+b}} \cdot \left( a^{a}b^{b} \right)^{\frac{(a+b-1)}{(a+b)^2}n} \cdot \left( \prod_{v\in V(G)} \frac{d_G(v) + C}{e} \right)^{1 - \frac{1}{a+b}}. \nonumber
    \end{align}
    The inequality \eqref{eq:use-holder-tree} follows from H\"older's inequality, while \eqref{eq:use-partition-tree-sum} follows from Lemma~\ref{lem:partition-tree-sum}. This completes the proof.
    
\end{proof}


\section{Multigraph Kahn--Lov\'{a}sz theorem}\label{sec:multigraph}

In this section, we prove Theorem~\ref{thm:two-cycles} by using a multigraph analogue of the Kahn--Lov\'{a}sz theorem, which is as follows. We regard two perfect matchings as distinct if they consist of different edge copies, even when they have the same underlying simple matching.

\begin{theorem}\label{thm:multigraph-KL}
    Let $M$ be an $n$-vertex loopless multigraph without isolated vertices. For each $v\in V(M)$, let $\mu_v\defeq \max\{\mu(uv): u\in V(M)\}$ and let $d_v$ be a real number satisfying $d_M(v)\leq d_v$.
    Then we have
    \begin{equation*}
        N_{\mathrm{factor}}(K_2; M) \leq \exp\left( \sum_{v\in V(M)} \frac{2\mu_v}{\sqrt{d_{v}}} \right) \cdot \prod_{v\in V(M)} \left( \frac{d_{v}}{e}\right)^{\frac{1}{2}}.
    \end{equation*}
\end{theorem}
We introduce $d_v$ instead of working directly with $d_M(v)$ because the function
$
x \mapsto \frac{\log x}{2} + \frac{2\mu}{\sqrt{x}}
$
is neither globally increasing nor concave. This lack of monotonicity and concavity makes it difficult to directly apply estimates involving $d_M(v)$ to the Kruskal--Katona-type problem.

We prove Theorem~\ref{thm:multigraph-KL} using the entropy method, following the approach pioneered by Radhakrishnan~\cite{Radhakrishnan} in his proof of Theorem~\ref{thm:bregman-minc}. Our argument is also inspired by the work of Linial and Luria on counting Steiner triple systems and perfect matching decompositions of complete graphs~\cite{Linial-Luria}, as well as Luria's upper bound on the number of perfect matchings in regular hypergraphs~\cite{Luria}.

\begin{proof}[Proof of Theorem~\ref{thm:multigraph-KL}]
    Let $\Psi$ denote the set of perfect matchings in $M$. We plan to estimate the entropy of a perfect matching $\psi$ chosen uniformly at random from $\Psi$. For each vertex $v\in V(M)$, let $Z_v$ be the edge of $\psi$ that covers $v$. Then by the monotonicity of entropy, we have
    \begin{equation*}
        H(\psi) \leq H(Z_v: v\in V(M)).
    \end{equation*}
    Assign independently to each vertex $v\in V(M)$ a random label $\sigma(v)$ uniformly distributed on $[0,1]$. Almost surely, these labels are distinct and hence induce an ordering of $V(M)$.
    Then, by the chain rule for entropy applied to $H(Z_v: v\in V(M))$, we have
    \begin{equation*}
        H(\psi) \leq \sum_{v\in V(M)} H(Z_v| Z_{v'}\text{, $\sigma(v') > \sigma (v)$}).
    \end{equation*}
    Taking expectations over $\sigma$ gives,
    \allowdisplaybreaks
    \begin{align}
        H(\psi) &\leq \mathbb{E}_{\sigma}\left[ \sum_{v\in V(M)} H(Z_v| Z_{v'}\text{, $\sigma(v') > \sigma (v)$}) \right] \nonumber \\
        &= \mathbb{E}_{\sigma}\left[ \sum_{v\in V(M)} \mathbb{E}_{(Z_{v'} \text{, $\sigma(v') > \sigma (v)$})}[ H(Z_v| Z_{v'} = z_{v'}\text{, $\sigma(v') > \sigma (v)$})] \right]. \label{eq:chain}
    \end{align}

    We now fix $\sigma$ and $v\in V(M)$.
    We denote by $N_v$ the number of possible edges to be $Z_v$ under the given previous choices. Then by the maximality of the uniform entropy and \eqref{eq:chain}, we have
    \begin{equation*}
        H(\psi) \leq \mathbb{E}_{\psi} \left[ \sum_{v\in V(M)} \mathbb{E}_{\sigma}[ \log (N_v)] \right].
    \end{equation*}
    Assume an edge $e$ whose endpoints are $u$ and $v$ such that $uv\in E(\psi)$ and $\sigma(u) > \sigma(v)$. Then, when we reach the time that exposes $v$, we already know that $Z_v = e$. Denote $O_v$ by the event that $\sigma(u) < \sigma(v)$. We observe that since $\sigma$ was chosen uniformly at random, if we restrict $\sigma(v)$ to a fixed value, then $\Pr[O_v] = \sigma(v)$. Thus, we have
    \allowdisplaybreaks
    \begin{align}
         H(\psi) &\leq \mathbb{E}_{\psi} \left[ \sum_{v\in V(M)} \mathbb{E}_{\sigma}[ \log (N_v)] \right] \nonumber \\
         &= \mathbb{E}_{\psi} \left[ \sum_{v\in V(M)} \mathbb{E}_{\sigma(v)} \mathbb{E}_{\sigma|\sigma(v)}[\Pr[O_v]\cdot \log (N_v)] \right] \nonumber \\
         &= \mathbb{E}_{\psi} \left[ \sum_{v\in V(M)} \mathbb{E}_{\sigma(v)}[ \sigma(v) \cdot \mathbb{E}_{\sigma|\sigma(v)}[\log (N_v)] ]\right] \nonumber \\
         &\leq \mathbb{E}_{\psi} \left[ \sum_{v\in V(M)} \mathbb{E}_{\sigma(v)}[ \sigma(v) \cdot \log (\mathbb{E}_{\sigma|\sigma(v)}[N_v])] \right] \nonumber \\
         &= \mathbb{E}_{\psi}\left[ \sum_{v\in V(M)} \int_0^1 \sigma(v) \log\left(\mathbb{E}_{\sigma|\sigma(v)}[N_v]\right) d\sigma(v) \right]. \label{eq:after-jensen}
    \end{align}
    The penultimate inequality holds by Jensen's inequality.

    Let $e \in E(\psi)$ be an edge that covers $v$ and let $u$ be the other endpoint of $v$. Then every multiple edge of $uv$ is counted in $N_v$.
    Let $e'$ be an edge of $M$ whose endpoints are $u'$ and $v$, where $u' \neq u$, and let $u''$ be the other endpoint of the edge $Z_{u'}$.
    Observe that $e'$ is counted in $N_v$ only if $\sigma(u'), \sigma(u'') < \sigma(v)$. Thus, by linearity of expectation, it holds that
    \begin{equation}\label{eq:exp-Nv}
        \mathbb{E}_{\sigma|\sigma(v)}[N_v] = \mu(uv) + (d_M(v)-\mu(uv))\cdot \sigma(v)^2 \leq \mu_v + (d_M(v) - \mu_v)\cdot \sigma(v)^2.
    \end{equation}
    By combining \eqref{eq:after-jensen} and \eqref{eq:exp-Nv}, we have
    \allowdisplaybreaks
    \begin{align}
        H(\psi) &\leq \mathbb{E}_{\psi}\left[ \sum_{v\in V(M)} \int_0^1 \sigma(v) \log\left( \mu_v + (d_M(v) - \mu_v)\cdot \sigma(v)^2 \right) d\sigma(v) \right] \nonumber \\
        &= \sum_{v\in V(M)} \int_0^1 [x \log(\mu_v + (d_M(v) - \mu_v)x^2)] dx \nonumber \\
        &\leq \sum_{v\in V(M)} \int_0^1 [x \log(\mu_v + d_v x^2)] dx. \label{eq:integral-form}
    \end{align}
    For $a,b>0$, a direct calculation gives
    \begin{equation*}
        \int_0^1 [x \log(a + bx^2)]dx = \frac{(a+b)\log(a+b) - a\log a - b}{2b}.
    \end{equation*}
    Thus, for each $v\in V(M)$, we have
    \allowdisplaybreaks
    \begin{align}
         \int_0^1 [x \log(\mu_v + d_v x^2)] dx &= \frac{(d_v + \mu_v) \log(d_v + \mu_v) - \mu_v \log \mu_v}{2 d_v} - \frac{1}{2} \nonumber \\
         &\leq \frac{1}{2}\cdot \left( \left(1 + \frac{\mu_v}{d_v}\right)\left(\log d_v + \log\left(1 + \frac{\mu_v}{d_v}\right)\right) - 1 \right) \nonumber \\
         &\leq \frac{1}{2}\cdot \left( \log d_v + \frac{\mu_v}{\sqrt{d_v}} + \frac{2\mu_v}{d_v} - 1 \right) \nonumber \\
         &\leq \frac{1}{2} \cdot \log \left(\frac{d_v}{e} \right) + \frac{2\mu_v}{\sqrt{d_v}}. \label{eq:integral-estimate}
    \end{align}
    The penultimate inequality holds since $\log x \leq \sqrt{x}$ and $\log(1 + x) \leq x$ for all $x > 0$.

    Then by \eqref{eq:integral-form} and \eqref{eq:integral-estimate},
    \begin{equation*}
        H(\psi) \leq \frac{1}{2} \cdot \log \left( \prod_{v\in V(M)} \frac{d_v}{e} \right) + \sum_{v\in V(M)} \frac{2\mu_v}{\sqrt{d_v}}. 
    \end{equation*}
    
    Since $\psi$ is chosen uniformly, we have $\log|\Psi| = H(\psi)$. Thus, 
    \begin{equation*}
        N_{\mathrm{factor}}(K_2; M) = |\Psi| \leq \exp\left( \sum_{v\in V(M)} \frac{2\mu_v}{\sqrt{d_v}} \right) \cdot \prod_{v\in V(M)} \left( \frac{d_v}{e}\right)^{\frac{1}{2}}.
    \end{equation*}
    This completes the proof.
\end{proof}

As a direct corollary, we have the following.

\begin{corollary}\label{cor:multigraph-KK}
    For all $\mu \geq 1$ and  $\ve > 0$, there exists $C = C(\mu, \ve) > 0$ such that the following holds. Let $M$ be an $n$-vertex $m$-edge loopless multigraph where $m \geq C n$ and $\mu(uv) \leq \mu$ for all $u, v \in V(M)$. Then
    \begin{equation*}
        N_{\mathrm{factor}}(K_2; M) \leq \left( (1 + \ve) \frac{2m}{e n} \right)^{\frac{n}{2}}.
    \end{equation*}
\end{corollary}

\begin{proof}[Proof of Corollary~\ref{cor:multigraph-KK}]
    We fix parameters as 
    $$
        0 < \frac{1}{C} \ll \ve' \ll \frac{1}{\mu}, \ve \leq 1.
    $$

    If $M$ has an isolated vertex, then trivially we obtain the desired inequality. Thus, we may assume that $M$ has no isolated vertices.
    
    Let $f: \mathbb{R}_{\geq 1} \to \mathbb{R}$ be a function such that $f(x) = \frac{\log x - 1}{2} + \frac{2\mu}{\sqrt{x}}$.
    Denote $\Psi$ by the set of perfect matchings in $M$. Then by Theorem~\ref{thm:multigraph-KL} with $d_v = d_M(v) + 9\mu^2$ for each $v\in V(M)$, we have
    \begin{equation}\label{eq:f(dMv+9mu2)}
        \log |\Psi| \leq \frac{1}{2} \cdot \sum_{v\in V(M)} (\log (d_M(v) + 9\mu^2) - 1)  + \sum_{v\in V(M)} \frac{2\mu}{\sqrt{d_M(v) + 9\mu^2}} = \sum_{v\in V(M)} f(d_M(v) + 9\mu^2).
    \end{equation}

    Since $f''(x) = \frac{3\mu - \sqrt{x}}{2 x^2\sqrt{x}}$, the function $f$ is concave on the interval $[9\mu^2, \infty)$. Since $d_M(v) \geq 1$ and $\sum_{v\in V(M)} d_M(v)  = 2m$, the right-hand side of \eqref{eq:f(dMv+9mu2)} is maximized when we replace $d_M(v) + 9\mu^2$ by $2m/n + 9\mu^2$ by Jensen's inequality.
    As $2m / n \geq 2C$, which is sufficiently large, we have
    \allowdisplaybreaks
    \begin{align*}
        \log |\Psi| &\leq n \cdot f\left( \frac{2m}{n} + 9\mu^2 \right) \\
        &\leq \frac{n}{2} \log \left( \frac{2m/n + 9\mu^2}{e} \right) + \frac{2\mu n}{\sqrt{2m/n}} \\
        &\leq \frac{n}{2} \log \left((1 + \ve') \frac{2m}{en} \right) + \ve' n \\
        &\leq \frac{n}{2} \log\left( (1 + \ve) \frac{2m}{en} \right).
    \end{align*}
    Thus, we obtain that
    $
        |\Psi| \leq \left( (1 + \ve) \frac{2m}{e n} \right)^{n/2}.
    $
    This completes the proof.
\end{proof}

We are now ready to prove Theorem~\ref{thm:two-cycles}.

\begin{proof}[Proof of Theorem~\ref{thm:two-cycles}]
    Let $n, \ell \geq 2$ be integers such that $n$ is divisible by $2\ell$. Let $F$ be a connected graph on $2\ell$ vertices that contains $2 \cdot C_{\ell}$ as a spanning subgraph.
    Let $W_\ell$ be the graph obtained from two vertex-disjoint copies of $C_\ell$ by adding one edge joining the two cycles. Since $F$ is connected, $W_{\ell}$ is a spanning subgraph of $F$. 
    The case $\ell=2$ is almost identical. The only difference is that $|\Aut(W_2)|=2$ and $|\Aut(C_2)|=2$, whereas $|\Aut(W_{\ell})|=8$ and $|\Aut(C_{\ell})|=2\ell$ for every $\ell \geq 3$. We therefore prove only the case $\ell \geq 3$.

    We now fix parameters as follows.
    $$
        0 < \frac{1}{C} \ll \frac{1}{\ell}, \ve \leq 1.
    $$

    Let $G$ be an $n$-vertex graph on $m$ edges with $m \geq C n$. Let $\calC$ and $\calW$ be the set of $C_{\ell}$-factors and $W_{\ell}$-factors in $G$, respectively. 
    Since $C_{\ell}$ is obviously Hamiltonian and $|\Aut(C_{\ell})| = 2\ell$, by Corollary~\ref{cor:factor-kruskal-katona}, it holds that
    \begin{equation}\label{eq:calc-upper}
        |\calC| \leq \left( \frac{(1 + \ve)}{2^{1/(\ell - 1)}} \cdot \frac{2m}{en} \right)^{\left( 1 - \frac{1}{\ell} \right)n}.
    \end{equation}

    We fix a $C_{\ell}$-factor $H\in \calC$ and enumerate each $C_{\ell}$ in $H$ as $C_{\ell}^{(1)}, \dots, C_{\ell}^{(n/\ell)}$. 
    We define an auxiliary multigraph $M_H$ on the vertex set $\{C_{\ell}^{(1)}, \dots, C_{\ell}^{(n/\ell)}\}$ such that $\mu(C_{\ell}^{(i)}, C_{\ell}^{(j)})$ be the number of crossing edges between $C_{\ell}^{(i)}$ and $C_{\ell}^{(j)}$ for each distinct $i, j \in [n/\ell]$. Then the number of perfect matchings in $M_H$ is exactly the number of distinct ways to extend $H$ to a $W_{\ell}$-factor. Also, for a given $W_{\ell}$-factor, it is always obtained by this process by decomposing each $W_{\ell}$ in the $W_{\ell}$-factor into $2\cdot C_{\ell}$ and an edge. Thus, we summarize as follows.
    \begin{equation}\label{eq:ClWl-identity}
        |\calW| = \sum_{H\in \calC} N_{\mathrm{factor}}(K_2; M_H).
    \end{equation}

    By our construction of $M_H$, we observe that
    \begin{equation*}
        \frac{m}{2} \leq m - \binom{\ell}{2} \cdot \frac{n}{\ell} \leq |E(M_H)| \leq m.
    \end{equation*}
    Also, the multiplicity of each edge of $M_H$ is bounded above by $\ell^2$.
    Therefore, by Corollary~\ref{cor:multigraph-KK}, we have
    \begin{equation}\label{eq:pm-in-Mh}
        N_{\mathrm{factor}}(K_2; M_H) \leq \left( (1 + \ve) \frac{2 \ell m}{e n} \right)^{\frac{n}{2\ell}},
    \end{equation}
    for each $H\in \calC$.
    By combining \eqref{eq:calc-upper}, \eqref{eq:ClWl-identity}, and \eqref{eq:pm-in-Mh}, we have
    \allowdisplaybreaks
    \begin{align*}
         |\calW| &= \sum_{H\in \calC} N_{\mathrm{factor}}(K_2; M_H) \\
         &\leq |\calC| \cdot \left( (1 + \ve) \frac{2 \ell m}{e n} \right)^{\frac{n}{2\ell}} \\
         &\leq \left( \frac{(1 + \ve)}{2^{1/(\ell - 1)}} \cdot \frac{2m}{en} \right)^{\left( 1 - \frac{1}{\ell} \right)n} \cdot  \left( (1 + \ve) \frac{2 \ell m}{e n} \right)^{\frac{n}{2\ell}} \\
         &\leq 2^{-\frac{n}{\ell}} \cdot \ell^{\frac{n}{2\ell}} \cdot \left((1 + \ve) \frac{2m}{en} \right)^{\left( 1 - \frac{1}{2\ell} \right)n}.
    \end{align*}
    Thus, by Lemma~\ref{lem:aut-double-counting}, we have 
    \allowdisplaybreaks
    \begin{align*}
        N_{\mathrm{factor}}(F; G) &\leq \left( \frac{|\Aut(W_{\ell})|}{|\Aut(F)|} \right)^{\frac{n}{2\ell}} |\calW| \\
        &\leq |\Aut(F)|^{-\frac{n}{2\ell}} \cdot 2^{\frac{3n}{2\ell}} \cdot 2^{-\frac{n}{\ell}} \cdot \ell^{\frac{n}{2\ell}} \cdot \left((1 + \ve) \frac{2m}{en} \right)^{\left( 1 - \frac{1}{2\ell} \right)n} \\
        &\leq \left( \frac{2\ell}{|\Aut(F)|} \right)^{\frac{n}{2\ell}} \cdot \left((1 + \ve) \frac{2m}{en} \right)^{\left( 1 - \frac{1}{2\ell} \right)n} \\
        &= \left[(1 + \ve) \left( \frac{|V(F)|}{|\mathrm{Aut}(F)|} \right)^{\frac{1}{|V(F)| - 1}} \frac{2m}{e n} \right]^{\left(1 - \frac{1}{|V(F)|}\right) n}.
    \end{align*}
    This completes the proof.
\end{proof}


\section{Concluding Remarks}\label{sec:concluding}
In this paper, we study the number $N_{\mathrm{factor}}(F; G)$ for connected graphs $F$. In particular, we establish an asymptotically sharp Kahn--Lov\'{a}sz-type inequality for every Hamiltonian graph $F$. As a direct consequence, we also obtain asymptotically sharp Kruskal--Katona-type inequalities for $F$-factors.
On the other hand, as shown in Section~\ref{sec:lowerbounds}, the behavior can be substantially different for certain connected non-Hamiltonian graphs $F$. Although Theorem~\ref{thm:general-factors} provides a Kahn--Lov\'{a}sz-type inequality for every connected graph $F$, there remains a considerable gap between the upper and lower bounds in general, both for the degree-sequence setting and for the corresponding Kruskal--Katona-type problem. This naturally leads to the following problem.

\begin{problem}\label{prob:kruskal-katona-type}
    Determine $N_F(n,m)$ up to a subexponential multiplicative factor for every connected graph $F$. 
\end{problem}

We note that the asymptotically sharp lower-bound constructions for Theorem~\ref{thm:main-factors} and Corollary~\ref{cor:factor-kruskal-katona}, which concern Hamiltonian graphs $F$, are given by disjoint unions of cliques, as described in Section~\ref{sec:lowerbounds}. We believe that the same phenomenon should hold more generally for every connected graph $F$ admitting a perfect fractional matching. We conclude the paper with the following conjecture.

\begin{conjecture}\label{conj:KL-pfm}
    For every $\ve > 0$ and a connected graph $F$ admitting a perfect fractional matching, there exists a constant $C = C(F, \ve) > 0$ such that the following holds. For all $n$-vertex graph $G$, we have
    \begin{equation*}
        N_{\mathrm{factor}}(F; G) \leq \left( (1  + \ve) \frac{|V(F)|}{|\mathrm{Aut}(F)|} \right)^{\frac{n}{|V(F)|}} \cdot \left( \prod_{v\in V(G)} \frac{d_G(v) + C}{e} \right)^{1 - \frac{1}{|V(F)|}}.
    \end{equation*} 
\end{conjecture}


\vspace{-0.2cm}

\providecommand{\MR}[1]{}
\providecommand{\MRhref}[2]{%
  \href{http://www.ams.org/mathscinet-getitem?mr=#1}{#2}
}

    \bibliographystyle{amsplain_initials_nobysame}
    \addcontentsline{toc}{section}{References}
    \bibliography{bibfile}


\end{document}